\documentclass[11pt, english]{smfart}
\usepackage[T1]{fontenc}
\usepackage[english,francais]{babel}
\usepackage{graphicx}
\usepackage{hyperref}
\usepackage{amssymb}
\usepackage{mathtools}
\usepackage{tikz}
\usepackage{tikz-cd} 
\usepackage{amsmath}
\usetikzlibrary{matrix}
\usepackage{smfthm}
\usepackage{mathrsfs}
\usepackage{pb-diagram}
\usepackage{yfonts}
\usepackage[inline]{enumitem}
\usepackage{color}
\usepackage{verbatim}
\usepackage{bbm}
\usetikzlibrary{matrix}
\def\R{\mathbb{R}}
\def\Q{\mathbb{Q}}

\def\N{\mathbb{N}}
\def\O{\mathscr{O}}
\def\L{\Lambda}

\def\nm{\lVert\cdot\rVert}

\def\deg{{\mathrm{deg}}}
\def\vol{\widehat{\mathrm{vol}}}

\def\mumin{\widehat{\mu}_{\min}}
\def\mumax{\widehat\mu_{\max}}

\def\limnto{\lim\limits_{n\rightarrow +\infty}}

\def\ot{\otimes}
\def\shfF{\mathscr{F}}
\def\shfG{\mathscr{G}}
\def\desF{\mathscr{F}_{\mathrm{des}}}
\def\shfL{\mathscr{L}}
\def\shfM{\mathscr{M}}
\def\shO{\mathscr{O}}

\def\ardeg{\widehat{\mathrm{deg}}}

\def\Td{\mathrm{Td}}
\def\ch{\mathrm{ch}}
\def\degH{\mathrm{deg}_{\mathcal{H}}}
\def\rank{\mathrm{rk}}
\def\ovl{\overline}
\def\hdX{\mathfrak{X}}
\def\rig{\mathrm{rig}}
\def\scrX{\mathscr{X}}

\def\scrA{\mathscr{A}}

\def\hd{\mathfrak}
\newcommand{\rest}[2]{\left.{#1}\right\vert_{{#2}}}

\makeatletter
\newcommand\tint{\mathop{\mathpalette\tb@int{t}}\!\int}
\newcommand\bint{\mathop{\mathpalette\tb@int{b}}\!\int}
\newcommand\tb@int[2]{%
  \sbox\z@{$\m@th#1\int$}%
  \if#2t%
    \rlap{\hbox to\wd\z@{%
      \hfil
      \vrule width .35em height \dimexpr\ht\z@+1.4pt\relax depth -\dimexpr\ht\z@+1pt\relax
      \kern.05em 
    }}
  \else
    \rlap{\hbox to\wd\z@{%
      \vrule width .35em height -\dimexpr\dp\z@+1pt\relax depth \dimexpr\dp\z@+1.4pt\relax
      \hfil
    }}
  \fi
}
\makeatother
\begin{document}
\title{A relative bigness inequality and equidistribution theorem over function fields}
\author{Wenbin Luo\thanks{
This work was supported by JSPS KAKENHI Grant Number JP20J20125.}}
\email{w.luo@math.kyoto-u.ac.jp}
\address{Department of Mathematics, Graduate School of Science, Kyoto university, Kyoto 606-8502, Japan}
\date{\today}
\begin{abstract}For any line bundle written as a subtraction of two ample line bundles, Siu's inequality gives a criterion on its bigness. We generalize this inequality to a relative case. The arithmetic meaning behind the inequality leads to its application on algebraic dynamic systems, which is the equidistribution theorem of generic and small net of subvarieties over a function field.
\end{abstract} 

\maketitle
\tableofcontents

\section{Introduction}
\subsection{Background}
The equidistribution theorem for generic small points is a crucial part in the proof of the Bogomolov conjecture. In the case of number fields, it was proved by Szpiro, Ullmo and Zhang \cite{SUZ,ullmo1998positivite,zhang1998equidistribution} by using the arithmetic Hilbert-Samuel theorem. In \cite{moriwaki2000arithmetic}, Moriwaki introduced the height theory over finitely generated fields, based on which he proved an arithmetic relative version of equidistribution theorem. In \cite{Yuan_2008}, Yuan reveal the relationship between arithmetic bigness and equidistribution theorem over number fields. In the non-Archimedean case, such a equidistribution can be considered with respect to the Chamber-Loir measures \cite{Chambert_Loir_2006}. Gubler and Faber independently transferred the result of Yuan to the case over function fields \cite{gubler2008equidistribution,Faber_2009}. In this paper, we prove an equivalent result by implementing a relative version of Siu's inequality on bigness. 
\subsection{Relative Siu's inequality}
Let $X$ be a projective variety of dimension $d$, and $L$ and $M$ be nef line bundles. Then $$\mathrm{vol}(L\ot M^\vee)\geq c_1(L)^d-d\cdot c_1(L)^{d-1}c_1(M)$$ 
due to Siu's inequality \cite[Theorem 2.2.15]{Positivity}.
It is very natural to consider its relative version. An essential problem is to find a good analogy of $h^0(\cdot)$. Let $k$ be a field of characteristic $0$. We fix a projective and normal $k$-variety $Y$ of dimension $e>0$, and a family of ample line bundles $\mathcal H=\{H_1,\dots,H_{e-1}\}$ on $Y$. For any coherent sheaf $\shfF$ on $Y$, we define the degree $\degH(\shfF)$ of $\shfF$ by $$\degH(\shfF):=c_1(\shfF)\cdot c_1(H_1) \cdots c_1(H_{e-1}).$$
Then we can give an analogous bigness inequality
\begin{theo}[cf. Theorem \ref{ineq_rel_big}]
Let $\pi:\scrX\rightarrow Y$ be a projective and surjective morphism of projective $k$-varieties, where $\scrX$ is of dimension $d+e(d>0)$.  For any cycle $\alpha$ on $\scrX$, we denote by $\alpha\cdot\pi^*(\mathcal H)$ the intersection $$\alpha\cdot c_1(\pi^*H_1)\cdots c_1(\pi^*H_{e-1}).$$ 
Let $\mathscr E$, $\shfL$ and $\shfM$ be line bundles over $\scrX$ such that $\shfL$ and $\shfM$ are nef.
Then it holds that
$$
    \begin{aligned}
        \degH(&\pi_*(\mathscr{E}\otimes (\shfL\ot \shfM^\vee)^{\ot n}))\geq \\
             &\frac{c_1(\shfL)^{d+1}\cdot\pi^*(\mathcal{H})-(d+1)c_1(\shfL)^{d}c_1(\shfM)\cdot\pi^*(\mathcal{H})}{(d+1)!}n^{d+1}+o(n^{d+1}).
    \end{aligned}
$$
\end{theo}
\subsection{Intersection theory and Chambert-Loir measure}
Let $Y^{(1)}$ be the set of points of codimension $1$ in $Y$. 
For any $\omega\in Y^{(1)}$, given by the order $\mathrm{ord}_{\omega}(\cdot)$, we can construct a non-Archimedean absolute value $\lvert\cdot\rvert_\omega$ on the function field $K:=k(Y)$. Let $K_\omega$ denote the completion of $K$ with respect to $\lvert\cdot\rvert_\omega$. Let $X$ be a $K$-projective variety of dimension $d$. 
We denote by $X^{\mathrm{an}}_\omega$ the analytification of $X_\omega:=X\times \mathrm{Spec}K_\omega$ with respect to $\lvert\cdot\rvert_\omega$ in the sense of Berkovich. An adelic line bundle is a pair $\overline L=(L,\{\lvert\cdot\rvert_{\varphi,\omega}\}_{\omega\in Y^{(1)}})$ where $L$ is a line bundle on $X$ and $\lvert\cdot\rvert_{\varphi,\omega}$ is a continuous metric of the analytification of $L_\omega:=L\otimes K_\omega$ on $X^{\mathrm{an}}_\omega$. An adelic line bundle $\overline L$ is said to be semipositive if there exists a sequence $\{(\pi_n:\scrX_n\rightarrow Y,\shfL_n,l_n)\}$ where each $(\scrX_n,\shfL_n)$ is a $Y$-model of $(X,L^{\ot l_n})$ with $\shfL_n$ being relatively nef, such that $\overline L$ can be viewed as a limit of $\shfL_n^{\ot 1/l_n}$ as described in subsection \ref{subsect_adelic_line_bund}. We say $\overline L$ is integrable if it is a subtraction of two semipositive ones. We can thus define an intersection theory for integrable adelic line bundles by taking a limit of classical intersections (see subsection \ref{subsect_arith_inter}). For a semipositive adelic line bundle $\overline L$, this intersection theory allows us to define the Chambert-Loir measure $\widehat c_1(\overline L)^d_\omega$ on each $X^{\mathrm{an}}_\omega$, which can be normalized to a probability measure $\mu_{\overline L,\omega}$. If $Z$ is a closed subvariety of $X$, then we can similarly construct a restricted probability measure $\mu_{\overline L, Z,\omega}$ supported on $Z^{\mathrm{an}}_\omega$. 
\subsection{Equidistribution of subvarieties}
Let $I$ be a directed set. A \textit{net} of subvarieties $\{Z_\iota\}_{\iota\in I}$ is a set indexed by $I$. We say it is generic if the subvarieties do not accumulate in any proper closed subset of $X$. We say $\{Z_\iota\}_{\iota\in I}$ is small if the heights $h_{\overline L}(Z_\iota)$ converge to $0$ (see subsection \ref{subsect_arith_inter} for the definition of height functions). Then our relative bigness inequality gives the following equidistribution theorem:
\begin{theo}[cf. Theorem \ref{theo_equidistribution}] Let $\overline L$ be an adelic line bundle such that $\overline L$ is a limit of nef line bundles, $L$ is ample and $\widehat{c}_1(\overline L)^{d+1}=0$.
Let $\{Z_\iota\}_{\iota\in I}$ be a generic net of subvarieties of $X$ such that $\lim\limits_{\iota\in I} h_{\overline L}(Z_\iota)=0.$ Then for any $\omega\in Y^{(1)}$, $\{Z_\iota\}$ is equidistributed in the sense that $$\int_{X_\omega^{\mathrm{an}}}f \mu_{\overline L,\omega}=\lim\limits_{\iota\in I} \int_{X_\omega^{\mathrm{an}}}f \mu_{\overline L,Z_\iota,\omega}$$
for any continuous function $f:X^{\mathrm{an}}_\omega\rightarrow \R$.
\end{theo}
\subsection{Adelic arithmetic over function fields}
Let $X$ be a normal projective variety. Let $\overline L$ be a semipositive adelic line bundle defined by a sequence $\{(\pi_n:\scrX_n\rightarrow Y,\shfL_n,l_n)\}$. As $\shfL_n$ converge to $\overline L$, each coherent sheaves $(\pi_n)_*\shfL_n$ can be viewed as an adelic vector bundle , which is roughly a vector space $H^0(X,L)$ equipped with a collection of norms parametrized by $Y^{(1)}$. For the details of adelic vector bundle, we refer the reader to \cite[\S 4]{adelic} or subsection \ref{subsect_def_adlcv}. For an adelic vector bundle, we can define its degree, slope, and so on, which coincide with geometric ones when the adelic vector bundle comes from a torsion-free coherent sheaf. A comparison result given by Theorem \ref{theo_comparison_lattice_sup} makes it possible to use some adelic results, especially Proposition \ref{prop_min_slop_nef} is frequently used in our proof of the relative bigness inequality. 
\subsection{Comparison with other results}
We can find that all proofs of equidistribution theorems in \cite{Yuan_2008,gubler2008equidistribution,Faber_2009} following the same discipline in essence. Actually, in a recent work of Yuan and Zhang \cite{yuan2021adelic}, they integrated the results over number fields and function fields of a regular curve. In section \ref{sec_adelic_fun}, we deal with function fields of higher dimensional varieties as a very particular case in the framework of Chen and Moriwaki\cite{adelic}. This article shows the relationship between the theory based on integral models and the adelic point of view. Moreover, the relative version of Siu's inequality gives a novel approach to the equidistribution theorem over function fields.
\subsection{Organization of the paper}
In section \ref{sec_pre_alg_geo}, we talk about some algebro-geometric results over a polarised projective varieties, which are preliminaries to the relative bigness inequality.
Section \ref{sec_adelic_fun} include some crucial adelic results over function fields. Section \ref{sec_arak_int} gives a description of arakelov intersection theory and height functions, which are closely related to the Chambert-Loir measure.
In section \ref{sec_rel_big}, we prove the relative bigness inequality.
Section \ref{sec_equi} contains a reproof of the equidistribution theorem over function fields.
\section*{Acknowledgement}
I wish to show my appreciation to my advisor Atsushi Moriwaki who proposed this project and gave a lot of valuable suggestions. I would like to thank Fanjun Meng for some helpful discussions in algebraic geometry.
\section{Notation and conventions}
\subsection{A polarised variety}\label{subsection_pol_var}Throughout this article, let $k$ be a field of characteristic $0$, $Y$ be an $e$-dimensional normal projective $k$-variety $(e>0)$. Let $K$ be the function field $k(Y)$ of $Y$. Denote by $Y^{(1)}$ the set of points of codimension $1$ in $Y$. For each $\nu\in Y^{(1)}$, let $\lvert\cdot\rvert_\nu$ denote the non-Archimedean absolute value on $K$ satisfying $$\lvert f\rvert_\nu:=e^{-\mathrm{ord}_\nu(f)}$$
for $f\in K\setminus\{0\}$.
We equip $Y$ with a polarization by a collection of ample line bundles $\mathcal H:=\{H_1,\dots,H_{e-1}\}$. For any cycle $Z$ of codimension $1$, we define the degree by $\degH(Z)=c_1(H_1)\cdots c_1(H_{e-1})\cdot Z$. In particular, for any $\nu\in Y^{(1)}$, we denote by $\degH(\nu)$ the degree $\degH(\overline{\{\nu\}})$.
Let $\pi:\scrX\rightarrow Y$ be a proper morphism of $k$-varieties. For any $\alpha\in A_k(\scrX)$, let $\alpha \cdot \pi^*\mathcal H$ denote the intersection $\alpha \cdot c_1(\pi^* H_1) \cdots c_1(\pi^* H_{e-1})\in A_{k-e+1}(\scrX)$. Note that if $k=e-1$, then $\alpha\cdot \pi^*\mathcal H=\degH(\pi_*\alpha)$.

\subsection{$Y$-model}
Let $X$ be a projective $K$-variety. We say $\scrX$ is a $Y$\textit{-model} of $X$, if there exists a surjective and projective morphism $\pi:\scrX\rightarrow Y$ of $k$-varieties such that $\scrX_K=X$, where $\scrX_K$ is the generic fiber of $\scrX$. Let $\shfL$ be a line bundle on $\scrX$, we denote by $\shfL_K$ the pull-back of $\shfL$ under $\scrX_K\rightarrow \scrX$. Let $L$ be a line bundle on $X$. We say $(\pi:\scrX\rightarrow X,\shfL)$ is a $Y$-model of $(X,L)$ if $\scrX_K=X$ and $\shfL_K=L$. We just say $(\scrX,\shfL)$ is a model of $(X,L)$ if there is no ambiguity.

\subsection{Horizontal and vertical divisors}\label{sect_vert_div}We say a prime Weil divisor $D$ on $\scrX$ is \textit{horizontal} if $\pi|_{D}:D\rightarrow Y$ is dominant. We say $D$ is \textit{vertical} if $\pi(D)$ is of codimension $1$ in $Y$. In general, for a Weil divisor $D$, we can consider a decomposition 
$$D=D_h+D_1+D_2$$
where $D_h$ is the horizontal part, $D_1$ is the vertical part, and $D_2$ is of components whose image is of codimension at least $2$. In this paper, $D_2$ is of little interest since we always consider the degree of the push-forward to $Y$. To be more precise, let $\alpha$ be a $e$-dimensional cycle on $\scrX$, then $$\degH(\pi_* (\alpha\cdot D))=\degH(\pi_*(\alpha\cdot(D_h+D_1))).$$

\section{Preliminaries on algebraic geometry}\label{sec_pre_alg_geo}

Throughout this section, let $\pi:\scrX\rightarrow Y$ be a surjective and projective morphism of projective $k$-varieties where $\scrX$ is of dimension $d+e$ ($d \in \N$). 

\subsection{Degree of coherent sheaves}\label{subsec_degree}  Let $U$ be the regular locus of $Y$. 
As $Y$ is normal, $U$'s complement is of codimension at least $2$, and hence $A_{e-1}(U)\simeq A_{e-1}(Y)$. For any coherent sheaf $\shfF$, we define its first Chern class by $c_1(\shfF):=c_1(\shfF|_U)$.
\begin{defi}
The degree of a coherent sheaf $\shfF$ with respect to $\mathcal{H}$ (degree for short in the following content) is defined  by 
$$\degH(\shfF):=c_1(H_1)\cdots c_1(H_{e-1})\cdot c_1(\shfF).$$
If the support of $\shfF$ is of dimension $e$, the slope is defined by $$\mu(\shfF):=\frac{\degH(\shfF)}{\mathrm{rk}(\shfF)}$$
where the $\mathrm{rk}(\shfF)$ is the rank of $\shfF$.
If the support $\shfF$ is of dimension $<e$, we define the slope by convention $\mu(\shfF):=+\infty$.
\end{defi}
\begin{rema}\label{rema_pos_tor}
For a torsion sheaf $\shfF$, namely a sheaf with $\dim \mathrm{Supp}(\shfF)<e$, its degree can be calculated by 
$$\degH(\shfF)=\sum\limits_{\omega\in Y^{(1)}}\mathrm{length}_{\O_{Y,\omega}}(\shfF_\omega)\cdot \degH(\omega)$$
Hence any torsion sheaf is of non-negative degree.
\end{rema}
 Similarly as the classical theory, we can consider the Harder-Narasimhan filtration, which is guaranteed by the existence of destabilizing sheaf. 
\begin{theo}Let $\shfF$ be a torsion-free coherent sheaf. There exists a unique non-trivial subsheaf $\desF$ of $\shfF$ such that
\begin{itemize}
    \item[\textnormal{(1)}] $\mu(\desF)$ is maximal among all subsheaves of $\shfF$.
    \item[\textnormal{(2)}] if there is a subsheaf $\shfG\subset\shfF$ such that $\mu(\shfG)=\mu(\desF)$ then $\shfG\subset\desF$. 
\end{itemize}
\end{theo}
\begin{proof}
See \cite[Lemma 1.3.5]{huy2010}.
\end{proof}
\begin{coro}
Let $\shfF$ be a torsion-free coherent sheaf. Then $\shfF/\shfF_{\mathrm{des}}$ is torsion-free or zero. In particular, if $\rank(\shfF)=1$, then $\shfF=\desF$.
\end{coro}
\begin{proof}
Suppose that $\shfF/\shfF_{\mathrm{des}}$ has the non-zero torsion part $T(\shfF/\shfF_{\mathrm{des}})$. Let $\shfF'\subset \shfF$ be the preimage of $T(\shfF/\shfF_{\mathrm{des}})$. Then $\shfF_{\mathrm{des}}\subsetneq \shfF'$ and $$\mu(\shfF')=\frac{\degH(\shfF_{\mathrm{des}})+\degH(\shfF'/\shfF_{\mathrm{des}})}{\mathrm{rk}(\shfF_{\mathrm{des}})}{>} \mu(\shfF_{\mathrm{des}}),$$
which contradicts to the maximality of $\desF$.
\end{proof}
\begin{defi}
We say that a coherent sheaf $\shfF$ is $\mu$-semistable if $\shfF=\shfF_{\mathrm{des}}$. Let $\shfF$ be a torsion-free coherent sheaf. A Harder-Narasimhan filtration of $\shfF$ is a chain of subsheaves
$$0=\shfF_0\subsetneq \shfF_1\subsetneq \shfF_2\subsetneq\cdots\subsetneq \shfF_l=\shfF$$
such that each $\shfF_i/\shfF_{i-1}(1\leqslant i\leqslant l)$ is a $\mu$-semistable sheaf and we have a sequence of descending rational numbers:
$$\mu(\shfF_1)>\mu(\shfF_2/\shfF_1)>\cdots >\mu(\shfF_l/\shfF_{l-1}).$$
\end{defi}

We can show that for any torsion-free coherent sheaf $\shfF$, its Harder-Narasimhan filtration exists uniquely \cite[Theorem 1.3.4]{huy2010}.
We define the \textit{maximal slope} $\mu_{\max}(\shfF)$ and the \textit{minimal slope} $\mu_{\min}(\shfF)$ as the the maximal and minimal ones in the descending sequence of slopes associated to the Harder-Narasimhan filtration. It is obvious that $\mu_{\max}(\shfF)=\mu(\desF)$ due to the construction of Harder-Narasimhan filtration. For the minimal slope, we can show the following:

\begin{prop}
Let $\shfF$ be a torsion-free coherent sheaf. Then $\mu_{\min}(\shfF)$ equals to the following two
\begin{itemize}
    \item[\textnormal{(1)}] $\inf\{\mu(\shfG)\mid \shfG\text{ is a quotient sheaf of }\shfF\}$
    \item[\textnormal{(2)}] $\inf\{\mu(\shfG)\mid \shfG\text{ is a torsion-free quotient sheaf of }\shfF\}$
\end{itemize}
\end{prop}
\begin{proof}
We first show that (1) and (2) are equal. Let $\shfG$ be a quotient sheaf of $\shfF$. If $\shfG$ is torsion, then $\mu(\shfG)=+\infty$. Hence we may assume that the support of $\shfG$ is of dimension $e$. Let $T(\shfG)\subset \shfG$ be the torsion part of $\shfG$. By the non-negativity of torsion sheaf, it follows that $\mu(\shfG/T(\shfG))\leqslant \mu(\shfG)$, which shows that (2) and (1) are equal. 
Now suppose that there exists a torsion-free quotient sheaf $\shfG$ such that $\mu(\shfG)<\mu_{\min}(\shfF)$. If $\shfF$ is $\mu$-semistable, then $\mu_{\min}(\shfF)=\mu_{\max}(\shfF)=\mu(\shfF)$. Let $\shfF'$ be the kernel of $\shfF\rightarrow \shfG$. Then $\shfF'$ is non-zero by our assumption. We can compute that
$$\begin{aligned}
\mu(\shfG)&=\frac{\degH(\shfF)-\degH(\shfF')}{\rank(\shfG)}=\frac{\mu(\shfF)\rank(\shfG)+\mu(\shfF)\rank(\shfG)-\mu(\shfF')\rank(\shfF)}{\rank(\shfG)}\\
&=\mu(\shfF)+(\mu(\shfF)-\mu(\shfF'))\frac{\rank(\shfF')}{\rank(\shfG)}\geqslant \mu(\shfF),
\end{aligned}$$
which is a contradiction.
For general cases, let $$0=\shfF_0\subsetneq \shfF_1\subsetneq \shfF_2\subsetneq\cdots\subsetneq \shfF_l=\shfF$$ be the Harder-Narasimhan filtration of $\shfF$. Denote by $\shfF'_i(0\leqslant i\leqslant l)$ the image of $\shfF_i$ in $\shfG$. Then each $\shfF_i'/\shfF'_{i-1}$ is a quotient sheaf of $\shfF_i/\shfF_{i-1}$. The $\mu$-semistability of $\shfF_i/\shfF_{i-1}$ yields that $\mu(\shfF_i'/\shfF'_{i-1})\geqslant \mu(\shfF_i/\shfF_{i-1})\geqslant \mu_{\min}(\shfF)$. Therefore $$\mu(\shfG)=\frac{\sum\limits_{1\leqslant i\leqslant l}\degH(\shfF_i'/\shfF'_{i-1})}{\sum\limits_{1\leqslant i\leqslant l}\rank(\shfF'_i/\shfF'_{i-1})}\geqslant \mu_{\min}(\shfF).$$
\end{proof}

\begin{coro}\label{coro_ineq_min_slop}
Let $\shfF$ be a torsion-free coherent sheaf with $\mu_{\min}(\shfF)\geqslant 0$. For any subsheaf $\shfF'\subset \shfF$, it holds that $\degH(\shfF')\leqslant \degH(\shfF)$.
\end{coro}
 
Note that all notions and results above can be generalized to the quotient category $\mathrm{Coh_{e,e-1}}(Y)=\mathrm{Coh}_e(Y)/\mathrm{Coh}_{e-2}(Y)$, where $\mathrm{Coh}_{e'}(Y)$ is the category whose objects are coherent sheaves of dimension$\leqslant e'$ \cite[1.6.2]{huy2010}. 
Then a coherent sheaf $\shfF$ with torsion part $T(\shfF)$ of codimension at least $2$ can be identified with $\shfF/T(\shfF)$.
Therefore we can even show that if $\shfF$ and $\shfG$ are torsion-free, then $\mu_{\min}(\shfF\otimes\shfG)=\mu_{\min}(\shfF)+\mu_{\min}(\shfG)$ due to the $\mu$-semistability of tensor products \cite[Theorem 3.1.4]{huy2010}.

\subsection{Asymptotic minimal slopes}
\begin{defi}
Let $\shfL$ be a line bundle on $\scrX$. We denote by $\mu^{\mathrm{asy}}_{\min}(\shfL)$ the limit inferior
$$\liminf_{n\rightarrow+\infty} \frac{\mu_{\min}(\pi_*\shfL^{\ot n})}{n}\in \R\cup \{\pm\infty\}$$
\end{defi}
\begin{prop}\label{prop_asymp_min_slop}
Let $\shfL$ be a relatively ample line bundle over $\scrX$. We have the following:
\begin{enumerate}
    \item[\textnormal{(a)}] The sequence $\displaystyle\Big\{\frac{\mu_{\min}(\pi_*\shfL^{\ot n})}{n}\Big\}$ converges in a number $\R\cup\{+\infty\}$
    \item[\textnormal{(b)}] If $\shfF$ is a coherent sheaf on $\scrX$ such that $\pi_*(\shfF\ot \shfL^{\ot n})$ is torsion free for every $n\in\N_+$, then 
    $$\liminf_{n\rightarrow +\infty}\frac{\mu_{\min}(\pi_*(\shfF\ot \shfL^{\ot n}))}{n}\geqslant \mu_{\min}^{\mathrm{asy}}(\shfL).$$
\end{enumerate}
\end{prop}
\begin{proof}
Due to \cite[Example 2.1.28]{Positivity}, we have the following:
$$\pi_*(\shfF\ot \shfL^{\ot m})\otimes \pi_*(\shfL^{\ot n})\rightarrow \pi_*(\shfF\ot \shfL^{\ot m+n})\text{ is surjective for }m,n\gg 0$$
Assume that $\pi_*(\shfF\ot \shfL^{\ot n})$ are torsion-free. Then we can consider the sequence
$$\{a_k:=\mu_{\min}(\pi_*(\shfF\ot \shfL^{\ot k}))\}_{k\in \N_+}.$$
The surjectivity above gives that $$a_m+\mu_{\min}(\pi_*(\shfL^{\ot n}))\leqslant a_{m+n}.$$ Note that if $\shfF=\O_\scrX$, then the sequence $\{\displaystyle\frac{\mu_{\min}(\pi_*(\shfL^{\ot n}))}{n}\}$ converges to a number in $\R\cup \{+\infty\}$ due to Fekete's subadditive lemma, which proves (a).
We can easily see that for any fixed $n\in \N_+$, we write $k=ln+r(l\in\N_+,0\leqslant r<n)$, it holds that
$$\liminf_{k\rightarrow +\infty}\frac{a_k}{k}\geqslant \lim_{l\rightarrow +\infty} \frac{a_r+l\mu_{\min}(\pi_*(\shfL^{\ot n}))}{ln+r}=\frac{\mu_{\min}(\pi_*(\shfL^{\ot n}))}{n}.$$
Therefore $\displaystyle{\liminf\limits_{k\rightarrow+\infty} \frac{a_k}{k}\geqslant \sup\limits_n\frac{\mu_{\min}(\pi_*(\shfL^{\ot n}))}{n}=\mu_{\min}^{asy}(\shfL)}.$
\end{proof}

\subsection{Algebraic family of cycles}\label{sect_alg_cyc}
For any point $\omega\in Y^{(1)}$, let $\scrX_\omega$ denote the special fiber $\scrX\times_Y \mathrm{Spec}(\O_{Y,\omega}/\hd m_{Y,\omega})$ over $\omega$. For any irreducible component $W$ of $\scrX_\omega$, $W$ can be also viewed as a prime vertical divisor of $\scrX$. In the following, we do not distinguish $W$ and its corresponding divisor unless specified otherwise. Let $\lambda_W$ be the multiplicity of 
$\scrX_\omega$ at
$W$. We denote by $[\scrX_\omega]$ the cycle $\sum \lambda_W[W]$. 

Let $U$ be the regular locus of $Y$. Then $\overline{\{\omega\}}\cap U\rightarrow U$ is a regular embedding. Consider a commutative diagram
\[\begin{tikzcd}
\pi^{-1}(\overline{\{\omega\}}\cap U) \arrow{r}{} \arrow[swap]{d}{\pi_\omega} & \pi^{-1}(U) \arrow{d}{\pi} \\
\overline{\{\omega\}}\cap U \arrow{r}{i} & U
\end{tikzcd}
\]
As defined in \cite[\S 6]{fulton2014intersection}, for any $k\geqslant 0$, we can consider a refined Gysin homomorphism $$\begin{aligned}
i^!:A_k(\pi^{-1}(U))&\rightarrow A_{k-1}(\pi^{-1}(\overline{\{\omega\}}\cap U))\\
\alpha&\mapsto \alpha_\omega:=\alpha\cdot [\pi^{-1}(\overline{\{\omega\}}\cap U)].
\end{aligned}
$$
By \cite[Proposition 6.2 (a)]{fulton2014intersection}, it holds that \[\pi_{\omega*}(\alpha_\omega)=(\pi_*(\alpha))_\omega\in A_{k-1}(\overline{\{\omega\}}\cap U).\] Especially if $k=e$, since $\mathrm{codim}(Y\setminus U)\geqslant 2$, we can consider the homomorphism
$$A_e(\scrX)\rightarrow A_e(\pi^{-1}(U))\rightarrow A_{e-1}(\overline{\{\omega\}}\cap U)\simeq A_{e-1}(\ovl{\{\omega\}})$$
which coincides with the map $(\alpha\in A_{e}(\scrX))\mapsto\pi_*(\alpha\cdot[\scrX_\omega])$.
Suppose that $\pi_*(\alpha)=\lambda\cdot[Y]$. Then $$\pi_*(\alpha\cdot [\scrX_\omega])=\pi_{\omega*}((\alpha|_{\pi^{-1}(U)})_\omega)=(\pi_*(\alpha|_{\pi^{-1}(U)}))_\omega=\lambda\cdot[\omega]$$ holds for any $\omega\in Y^{(1)}$. We can view this as a generalization of \cite[Proposition 10.2]{fulton2014intersection}. In fact, $\lambda$ is the degree of the image of $\alpha$ in $A_0(\scrX_K)$.

As an application, here we give a calculation which will be useful in our follwoing content.
Let $A$ be a line bundle on $Y$, $\{\shfL_i\}_{0\leqslant i\leqslant d}$ be line bundles on $\scrX$. Let $L_i:=(\shfL_i)_K$ for each $i$. Then we can calculate that
\begingroup
\allowdisplaybreaks
\begin{align*}
&c_1(\shfL_0\ot\pi^*A)\cdots c_1(\shfL_d\ot\pi^* A)\cdot\pi^*\mathcal H-c_1(\shfL_0)\cdots c_1(\shfL_d)\cdot\pi^*\mathcal H\\
&=\sum_{0\leqslant i\leqslant d}\left(\prod\limits_{j\leqslant i} c_1(\shfL_j\ot \pi^*A)\prod\limits_{j>i} c_1(\shfL_j)-\prod\limits_{j< i} c_1(\shfL_j\ot \pi^*A)\prod\limits_{j\geqslant i} c_1(\shfL_j)\right)\cdot\pi^*\mathcal H\\
&=c_1(\pi^*A)\sum_{0\leqslant i\leqslant d} c_1(\shfL_0\ot\pi^*A)\cdots c_1(\shfL_{i-1}\ot\pi^*A)c_1(\shfL_{i+1})\cdots c_1(\shfL_{d+1})\cdot\pi^*\mathcal H\\
&=\degH \left(c_1(A)\pi_*\left(\sum_{0\leqslant i\leqslant d} \prod\limits_{j\leqslant i-1} c_1(\shfL_j\ot \pi^*A)\prod\limits_{j\geqslant i+1} c_1(\shfL_j)\right)\right)\\
&=\degH(A)\sum_{0\leqslant i\leqslant d}c_1(L_0)\cdots c_1(L_{i-1}) c_1(L_{i+1})\cdots c_1(L_{d}).
\end{align*}
\endgroup
In particular, if $\shfL_0=\cdots=\shfL_{d}=\shfL$, then \begin{equation}\label{calc_inter_alg_cyc}
    c_1(\shfL\ot \pi^*A)^{d+1}\cdot\pi^*\mathcal H-c_1(\shfL)^{d+1}\cdot\pi^*\mathcal H=\degH(A)c_1(\shfL_K)^d(d+1).
\end{equation}

\section{Adelic theory over function fields}\label{sec_adelic_fun}
\subsection{A remainder on adelic curves}\label{subsect_def_adlcv}
Let $K$ be a field and $M_K$ be the set of all its absolute values. Let $(\Omega, \mathcal A, \nu)$ be a measure space where $\mathcal A$ is a $\sigma$-algebra on $\Omega$ and $\nu$ is a measure on $(\Omega,\mathcal A)$. If there exists a $\phi: \Omega\rightarrow M_K, \omega\mapsto \lvert\cdot\rvert_\omega$ such that for any $a\in K\backslash\{0\}$, the function $$(\omega\in\Omega)\mapsto \ln|a|_\omega$$ is $\mathcal A$-measurable and $\nu$-integrable, then we call the data $(K,(\Omega, \mathcal A, \nu),\phi)$ an \textit{adelic curve}. For each $\omega\in\Omega$, we denote by $K_\omega$ the completion of $K$ with respect to the absolute value $\lvert\cdot\rvert_\omega$.
Moreover, we say $S$ is a \textit{proper adelic curve} if $S$ satisfies the product formula 
$$\int_{\omega\in\Omega} \ln|a|_\omega \nu(d\omega)=0.$$

\begin{rema}
In this paper, we only consider the case that an adelic curve is given by our polarised variety $Y$. 
Let $\mathcal A$ be the discrete $\sigma$-algebra on $Y^{(1)}$, and $\nu_{\mathcal H}$ be the measure on $Y^{(1)}$ such that $\nu_{\mathcal H}(\nu)=\degH(\nu)$. Then $S(Y,\mathcal H):=(K,(Y^{(1)},\mathcal A,\nu_{\mathcal H}),\phi)$ is a proper adelic curve where $\phi$ maps $\omega$ to $\lvert\cdot\rvert_\omega$.
\end{rema}

We now introduce the notion of adelic vector bundles. 
Let $E$ be a vector space over $K$ of dimension $r$. A \textit{norm family} $\xi$ is a set of norms $\{\nm_{\omega}\}_{\omega\in\Omega}$ parametrized by $\Omega$, where each $\nm_\omega$ is a norm on $E_{K_\omega}:=E\ot_K K_\omega$. We introduce some notations:
\begin{enumerate}
    \item[\textnormal{(a)}] Let $E^\vee$ be the dual space of $E$. We denote by $\nm_{\omega,*}$ the operator norm on $E^\vee_{K_\omega}:=E^\vee\ot_K K_{\omega}$, that is,
    $$\lVert s\rVert_{\omega,*}:=\inf_{x\in E\setminus\{0\}}\frac{\lvert s(x)\rvert_\omega}{\lVert x\rVert_\omega},$$
    where $s\in E_{K_\omega}^\vee$.
    Then $\xi^\vee:=\{\nm_{\omega,*}\}_{\omega\in\Omega}$ is a norm family on $E^\vee$. 
    \item[\textnormal{(b)}] Let $F$ be a subspace of $E$. For each $\omega\in \Omega$, let $\nm_{F,\omega}$ be the restriction of $\nm_{\omega}$ on $F_{K_\omega}$. The restricted norm family $\{\nm_{F,\omega}\}_{\omega\in\Omega}$ on $F$ is denoted by $\xi|_F$.
    \item[\textnormal{(c)}] Let $G=E/F$. Since $G_{K_\omega}:=G\ot_K{K_\omega}=E_{K_\omega}/F_{K_\omega}$, we define $\nm_{E\twoheadrightarrow G,\omega}$ the quotient norm of $\nm_\omega$, that is,
    $$\lVert s\rVert_{E\twoheadrightarrow G,\omega}:=\inf_{x\in f_{K_\omega}^{-1}(s)}\lVert x\rVert_\omega$$
    where $s\in G_{K_\omega}$ and $f_{K_\omega}:E_{K_\omega}\rightarrow G_{K_\omega}$. We denote by $\xi_{E\twoheadrightarrow G}$ the norm family $\{\nm_{E\twoheadrightarrow G,\omega}\}_{\omega\in\Omega}$ on G.
    \item[\textnormal{(d)}] Let $\det E$ be the determinant of $E$, which is $\wedge^r E$. For each $\omega\in\Omega$, we define $\nm_{\det,\omega}$ as follows:
    $$\lVert s\rVert_{\det,\omega}:=\inf_{\substack{s=t_1\wedge \cdots\wedge t_r\\ t_i\in E_{K_\omega}}} \lVert t_1\rVert_\omega\cdots\lVert t_r\rVert_\omega$$
    where $s\in \det E_{K_\omega}=(\det E)\ot_K K_{\omega}$. Therefore $\{\nm_{\det,\omega}\}_{\omega\in\Omega}$ is a norm family on $\det E$.
\end{enumerate}

\begin{defi}\label{def_adl_vec_bun} Let $\overline E=(E,\xi)$ be a pair of a finite-dimensional vector space over $K$ and a norm family on it. 
We say the norm family $\xi$ is \textit{measurable} if the function $$(\omega\in\Omega)\mapsto\lVert s\rVert_\omega$$ is $\mathcal A$-measurable with respect to $\nu$ for any $s\in E$.
We say the norm family $\xi$ is \textit{upper dominated} if
$$\forall s\in E\backslash\{0\}, \lefteqn{\int_\Omega \ln\lVert s \rVert_\omega}\lefteqn{\hspace{1.2ex}\rule[ 3.35ex]{1.1ex}{.05ex}}
\phantom{\int_\Omega \ln\lVert s \rVert_\omega}\nu(d\omega)<+\infty.$$
Moreover, if both $\xi$ and $\xi^\vee$ are upper dominated, we say $\xi$ is \textit{dominated}. $\overline E$ is said to be an \textit{adelic vector bundle} over $S$ if $\xi$ is dominated, and both $\xi$  and $\xi^\vee$ are measurable.
\end{defi}
Note that for any adelic vector bundle $(E,\xi)$, its dual $(E^\vee,\xi^\vee)$ is an adelic vector bundle due to \cite[4.1.30]{adelic}. Moreover, we have the following proposition on the preservation of being an adelic vector bundle.
\begin{prop}
We assume that $\mathcal A$ is discrete or $K$ admits a countable subfield dense in every $K_\omega$. Let $\ovl E=(E,\xi)$ be an adelic vector bundle.
\begin{itemize}
    \item[\textnormal{(1)}] Let $F$ be a subspace of $E$. Then $(F,\xi|_F)$ and $(E/F,\xi_{E\twoheadrightarrow E/F})$ are adelic vector bundles.
    \item[\textnormal{(2)}] The determinant $\mathrm{det}(\overline E):=(\det E,\det \xi)$ is an adelic vector bundle of dimension $1$.
\end{itemize}
\end{prop}
\begin{proof}
See \cite[Proposition 4.1.32]{adelic}.
\end{proof}

\begin{defi}
Let $(E,\xi)$ be an adelic vector bundle over $S$. If $\dim_K E=1$, we define its \textit{Arakelov degree} as $$\widehat{\mathrm{deg}}(E,\xi)=-\int_\Omega \ln\lVert s\rVert_\omega \nu(d\omega){,}$$ where $s$ is a nonzero element of $E$. The definition is independent of $s$ because $S$ is proper.
If $(E,\xi)$ is an adelic vector bundle (not necessarily of dimension $1$) over $S$, we define the Arakelov degree of $(E,\xi)$ as
$$\ardeg(E,\xi):=\begin{cases}\ardeg(\mathrm{det}E,\mathrm{det}\xi)&\text{ if }E\not=0,\\\ 0&\text{ if }E=0.\end{cases}$$
The slope, the maximal slope and the minimal slope of $(E,\xi)$ are defined respectively as follows:
\[\begin{cases}
\widehat{\mu}(E,\xi)&:={\displaystyle \frac{\ardeg(E,\xi)}{\dim_K(E)},}\\[2ex]
\mumax(E,\xi)&:=\sup\limits_{0\not=F\subseteq G}\widehat\mu(F,\xi|_F),\\
\mumin(E,\xi)&:=\inf\limits_{E\twoheadrightarrow G\not=0}\widehat\mu(G,\xi_{E\twoheadrightarrow G}).
\end{cases}
\]
\end{defi}

\subsection{The adelic view of torsion-free coherent sheaves}\label{sec_adelic_view} 
Let $\mathscr{E}$ be a torsion-free coherent sheaf on $Y$ of rank $r$. Let $\eta$ be the generic point of $Y$.
For each $\omega\in Y^{(1)}$, we can also define the lattice norm $\nm_{\mathscr E,\omega}$ on $\mathscr E_\eta$ by
$$\lVert s\rVert_{\mathscr E,\omega}=\inf\{\lvert\lambda\rvert_\omega\mid\lambda\in K_\omega, \lambda^{-1}s\in \mathscr{E}\ot \widehat{\mathcal O}_{Y,\omega}\}$$

The pair $(\mathscr{E}_{\eta},\{\nm_{\mathscr{E},\omega}\}_{\omega\in Y^{(1)}})$ is an adelic vector bundle whose Arakelov degree is exactly $\degH(\mathscr{E})=c_1(H_1)\cdots c_1(H_{e-1})c_1(\mathscr{E})$. 
We may wonder whether the adelic definition of maximal and minimal slope coincides with geometric ones. Given a subspace $F$ of $\mathscr E_\eta$, the adelic subbundle $\overline F=(F,\{\nm_{\mathscr E,\omega}|_F\}_\omega)$ is uniquely determined by $F$. But when we consider a subsheaf $\shfF$ of $\mathscr E$, it can not be determined by its generic stalk $\shfF_\eta$. An example is the ideal sheaf $\O(-D)$ of some effective Cartier divisor $D$, which admits an injection $0\rightarrow \O(-D)\rightarrow \O_Y$. The quotient is a torsion sheaf $\O_D$. Therefore $\O(-D)$ can not be viewed as an adelic subbundle of $\O_Y$.
But we can prove the following theorem.
\begin{theo}
Let $\mathscr E$ be a torsion-free coherent sheaf on $Y$. Then for any quotient sheaf $\mathscr E\twoheadrightarrow \shfG$ with $\shfG$ torsion-free, the adelic vector bundle $\left(\shfG,\{\nm_{\shfG,\omega}\}\right)$ is uniquely determined by the generic stalk $\shfG_\eta$. 
\end{theo}

\begin{proof}
For any $\omega\in Y^{(1)}$, we have the commutative diagram: 
\[\begin{tikzcd}
\mathscr E_\eta\ot K_\omega \arrow{r}{} & \shfG_\eta\ot K_\omega \\
\mathscr E_\omega\ot \widehat\O_{Y,\omega} \arrow{r}{} \arrow{u}{} & \shfG_\omega\ot \widehat{\O}_{Y,\omega} \arrow{u}{}
\end{tikzcd}
\]
Both of the horizontal arrows are surjection.
For any $t\in \shfG_\eta\ot K_\omega$, let $$N(t):=\inf\{\lvert\lambda\rvert_{\omega}\mid \exists s\in \mathscr E_\eta\ot K_\omega,\text{ s.t. }\overline s=t,\lambda^{-1}s\in \mathscr E_\omega\ot \widehat \O_{Y,\omega}\}.$$  Note that $N(t)$ concides with the quotient norm of $\nm_{\mathscr E,\omega}$ induced by $\mathscr E_\eta\twoheadrightarrow \shfG_\eta$. Indeed, for any $\lambda \in K_\omega$ such that $\lambda^{-1}t\in \shfG_\omega\ot \widehat \O_{Y,\omega}$, there exists an element $s\in \mathscr E_\omega\ot \widehat \O_{Y,\omega}$, whose image $\overline s$ is  $\lambda^{-1}t$, hence $\overline {\lambda s}=t$ and $\lambda^{-1}(\lambda s)=s\in \mathscr{E}_\omega\ot\widehat \O_{Y,\omega}.$ Therefore $N(t)\leqslant \lVert t\rVert_{\shfG,\omega}$. The reverse inequality is obvious. 
\end{proof}

Note that for a torsion-free coherent sheaf $\mathscr E$, there is a bijection between the saturated subsheaves of $\mathscr E$ and vector subspaces of $\mathscr E_\eta$ \cite[Proposition 1.3.1]{moriwaki2008torsionfree}. Then the following is obvious.
\begin{coro}\label{coro_geo_adelic}
Let $\mathscr E$ be a torsion-free coherent sheaf. Let $\overline E$ be the associated adelic vector bundle $\left(\mathscr{E}_\eta,\{\nm_{\mathscr{E},\omega}\}_{\omega\in Y^{(1)}}\right)$. Then
$\mumin(\overline E)= \mu_{\min}(\mathscr E)$ and $\mumax(\overline E)= \mu_{\max}(\mathscr E).$
\end{coro}

\subsection{Adelic line bundles}
Let $X$ be a projective $K$-scheme. We denote by $X^{\mathrm{an}}_\omega$ the analytification of $X_{K_\omega}:=X\times_{\mathrm{Spec}K} \mathrm{Spec}K_\omega$ with respect to $\lvert\cdot\rvert_\omega$ in the sense of Berkovich spaces \cite{Berkovich}. As a set, $X^{\mathrm{an}}_\omega$ consists of pairs $x=(p\in X_{K_\omega}, \lvert\cdot\rvert_x)$ where $\lvert\cdot\rvert_x$ is an absolute value on the residue field $\kappa(p)$ whose restriction on $K_\omega$ is $\lvert\cdot\rvert_\omega$. We denote by $\widehat\kappa(x)$ the completion of $\kappa(p)$ with respect to $\lvert\cdot\rvert_x$. 

Let $L$ be a line bundle over $X$. We denote by $L^{\mathrm{an}}_\omega$ the anlytification of $L_\omega:=L\ot \O_{X_{K_\omega}}$ on $X^{\mathrm{an}}_\omega$. A metric $\varphi_\omega$ of $L^{\mathrm{an}}_\omega$ is a collection $\{\lvert\cdot\rvert_\varphi(x)\}_{x\in X^{\mathrm{an}}_\omega}$ where each 
$\lvert\cdot\rvert_\varphi(x)$ is a norm on $L\ot \widehat\kappa(x)$. We say $\varphi$ is continuous if for any regular section $s$ of $L_\omega$ on an open subset $U\subset X_{K_\omega}$, the function
$$(x\in U^{\mathrm{an}})\mapsto \lvert s\rvert_\varphi(x)$$ is continuous.

An \textit{adelic line bundle} is a pair $(L,\{\varphi_\omega\}_{\omega\in\Omega})$ where each $\varphi_\omega$ is a continuous metric on $L^{\mathrm{an}}_\omega$. In \cite[\S 6]{adelic}, Chen and Moriwaki put some conditions of measurability and integrability on the family $\{\varphi_\omega\}_{\omega\in\Omega}$ of continuous metrics. In this paper, the adelic line bundles that we will consider satisfy their conditions.

\subsection{Metrics given by a model}
In this subsection, we fix a projective $K$-variety $X$ and a line bundle $L$ over $X$.
For any $\omega\in Y^{(1)}$, we denote by $K^\circ_{\omega}$ the valuation ring of $K_\omega$. Let $(\pi:\scrX\rightarrow Y,\shfL)$ be a $Y$-model of $(X,L)$. 
Since $Y$ is a projective variety, we have a morphism $\mathrm{Spec}(K_{\omega}^\circ) \rightarrow Y$. Consider the base change $\scrX_{K_{\omega}^{\circ}}\rightarrow \mathrm{Spec}(K_{\omega}^\circ)$ of $\pi$, which is flat due to the surjectivity \cite[Proposition 4.3.9]{Liu}. Let $\hdX_{\omega}$ be the completion of $\scrX_{K_{\omega}^\circ}$ with respect to the special fiber. Then $\hdX_\omega$ is an admissible formal scheme over $\mathrm{Spf}(K_\omega^\circ)$ due to the flatness \cite[\S 1, p.297]{BoschLutkebohmert}. Therefore we can consider its generic fiber $\hdX_{\omega,\mathrm{rig}}$ \cite[\S 4]{BoschLutkebohmert}. On the other hand, we can obtain a rigid space $X_{\omega}^{\mathrm{rig}}$ which is the rigid analytification of $X_\omega$. We have a natural morphism of rigid spaces, 
$$\hdX_{\omega,\mathrm{rig}}\rightarrow X_{\omega}^{\mathrm{rig}}$$
which is an isomorphism since $\scrX_{K_{\omega}^{\circ}}\rightarrow \mathrm{Spec}(K_{\omega}^\circ)$ is proper \cite[Theorem 5.3.1]{conrad1999irreducible}.
Let $\hd L_{\omega}$ be the completion of $L$. Let $L_\omega^\rig$ denote analytification of $L_\omega$ on $X^\rig_\omega$, $\hd L_{\omega,\rig}$ denote the generic fiber of $\hd L_\omega$. As defined in \cite[Lemma 7.4]{gubler1998local}, we can assign a formal metric $\lvert\cdot\rvert_{\shfL,\omega}$ on $L^\rig_{\omega}\simeq \hd L_{\omega,\rig}$, which can be extended to $L^{\mathrm{an}}_{\omega}$ on the Berkovich space $X^{\mathrm{an}}_\omega$ since the image of $X^\rig_\omega\rightarrow X^{\mathrm{an}}_\omega$ is dense. Such a formal metric is continuous. The following proposition shows that the action of pull-back on the formal metric only depends on the generic fiber.
\begin{prop}\label{prop_metric_puck}
Let $\rho:\scrX'\rightarrow \scrX$ be a proper morphism of projective $k$-varieties such that $\pi'=\pi\circ \rho:\scrX'\rightarrow Y$ is surjective. Let $\rho_K:\scrX'_K\rightarrow \scrX_K$ be the induced morphism on the generic fibers. 
Let $\shfL$ be a line bundle on $\scrX$. Then for each $\omega\in Y^{(1)}$,
$$\rho^*_K\lvert\cdot\rvert_{\shfL,\omega}=\lvert\cdot\rvert_{\rho^*\shfL,\omega}$$ as metrics on $\rho_K^*\shfL_K.$
\end{prop}
\begin{proof}
Since for each $\omega\in Y^{(1)}$, the model metric $\lvert\cdot\rvert_{\shfL,\omega}$ is determined by $\scrX_{\O_{Y,\omega}}\rightarrow \mathrm{Spec}\O_{Y,\omega}$ and $\shfL\ot \O_{Y,\omega}$, the proof is almost the same with the case that $Y$ is a regular projective curve. We refer the reader to \cite[Lemma 2.4]{Faber_2009}.
\end{proof}

\subsection{Pushforwards of adelic line bundles}\label{subsec_push_adelic_line}
Let $X$ be a projective $K$-variety and $L$ be a line bundle over $X$. 
Let $(\pi:\scrX\rightarrow Y,\shfL)$ be a model of $(X,L)$, which induces an adelic line bundle $\ovl L=(L,\{\lVert\cdot\rVert_{\shfL,\omega}\})$ on $X$.
For any $s\in H^0(X,L)$, we define the supremum norm as $$\lVert s\rVert_{\shfL,\omega}:=\sup\limits_{x\in X_{\omega}^{\mathrm{an}}}\lvert s\rvert_{\shfL,\omega}(x).$$
On the other hand, since $H^0(X,L)\simeq (\pi_*\shfL)_{\eta}$, we can consider the lattice norm $\lVert\cdot\rVert_{\pi_*\shfL,\omega}$ as well. 
The following comparison between the lattice norm and supremum norm is given in \cite[Lemma 6.3 and Theorem 6.4]{boucksom2021spaces}.
\begin{theo}\label{theo_comparison_lattice_sup}
There exists a constant $C_\omega$ depends only on $\scrX_{K^\circ_\omega}$ such that 
$$\nm_{\shfL,\omega}\leqslant \nm_{\pi_*\shfL,\omega}\leqslant C_\omega\nm_{\shfL,\omega}$$
Moreover, $C_\omega=1$ if the special fiber of $\scrX_{K^\circ_\omega}$ is reduced.
\end{theo}

Note that the pair $(H^0(X,L),\{\nm_{\shfL,\omega}\})$ is an adelic vector bundle due to \cite[Theorem 6.2.18]{adelic}. We will see that the comparison theorem leads to an application on degrees and minimal slopes.
\begin{prop}\label{prop_min_slop_inv} For each $n\in\N_+$, let $\ovl E_n$ be the adelic vector bundle $(H^0(X,L^{\ot n}),\{\nm_{\shfL^{\ot n},\omega}\}_{\omega\in Y^{(1)}})$. There exists a constant $C_\scrX$ depends on $\scrX$ only such that 
\[\begin{cases}
    \ardeg(\ovl E_n)\geqslant \degH(\pi_*(\shfL^{\ot n}))
    \geqslant \ardeg(\ovl E_n)-C_\scrX h^0(X,L),\\
    \mumin(\ovl E_n)\geqslant \mumin(\pi_*(\shfL^{\ot n}))\geqslant \mumin(\ovl E_n)-C_{\scrX}.
\end{cases}\]
In particular, $$\displaystyle
\mu_{\min}^{\mathrm{asy}}(\shfL)=\liminf_{n\rightarrow +\infty}\frac{\mumin(\ovl E_n)}{n}.$$
\end{prop}
\begin{proof}
Since the generic fiber $X$ is reduced, there exists an open subset $U\subset Y$ such that for any $\omega\in Y^{(1)}\cap U$, the special fiber of $\scrX_{K^\circ_\omega}$ is reduced. By Theorem \ref{theo_comparison_lattice_sup},
there exists a constant $C_\omega$ depends only on $\scrX_{K^\circ_\omega}$ such that 
$$\nm_{\shfL^{\ot n},\omega}\leqslant \nm_{\pi_*(\shfL^{\ot n}),\omega}\leqslant C_\omega\nm_{\shfL^{\ot n},\omega},$$ and $C_\omega=1$ for all but finitely many $\omega$. We are done by setting $C_\scrX=\sum_{\omega\in Y^{(1)}}\degH(\omega)\ln C_\omega$.
\end{proof}

\section{Arakelov intersection theory and height functions}\label{sec_arak_int}
In this section, we fix a projective $K$-variety $X$.
\subsection{Semipositive adelic line bundles}\label{subsect_adelic_line_bund}
Recall that an adelic line bundle $\overline L=(L,\{\phi_\omega\}_{\omega\in Y^{(1)}})$ consist of a line bundle $L$ over $X$ and a family of continuous metrics satisfies certain dominancy and measurablity properties. Consider two adelic line bundles $\overline L=(L,\{\phi_\omega\})$ and $\overline L'=(L,\{\phi_\omega'\})$ whose underlying line bundles coincide. For each $\omega\in Y^{(1)}$, note that the local distance $d(\phi_\omega,\phi_\omega')$ is set to be the maximum of a continuous function on $X^{\mathrm{an}}_{\omega}$ locally given by $x\mapsto \left\lvert \ln\frac{\lvert s\rvert_{\phi_\omega}(x)}{\lvert s\rvert_{\phi_\omega'}(x)}\right\rvert$ with $s$ is a rational section of $L_\omega$ not vanishing at $x$. This is independent of the choice of $s$.

In the case over function fields, the global distance can be given by a summation
$$d(\overline L,\overline L'):=\sum\limits_{\omega\in Y^{(1)}}\degH(\omega)d(\phi_\omega,\phi_\omega').$$
Now assume that there exists a sequence of triples $\{(\pi_n:\scrX_n\rightarrow  Y,\shfL_n,l_n)\}_{n\in\N_+}$ where each $(\scrX_n,\shfL_n)$ is a $Y$-model of $(X,L^{\ot l_n})$ for some positive integer $l_n$. Let $\overline L_{n}:=(L,\{\lvert\cdot\rvert_{\shfL_n,\omega}^{1/l_n}\})$.
We say an adelic line bundle $\overline L$ semipositive if
\begin{itemize}
    \item[\textnormal(i)] $\shfL_n$ are relatively nef,
    \item[\textnormal(ii)] $\lim\limits_{n\rightarrow \infty}d(\overline L, \overline L_n)=0$.
\end{itemize}
We say $\overline L$ is integrable if $\overline L=\overline L_1\otimes \overline L_2^{\vee}$ for semipositive line bundles $\overline L_1$ and $\overline L_2$.

\subsection{Arithmetic Intersection theory}\label{subsect_arith_inter} Here we give a reminder on the intersection theory described in \cite{Chambert_Loir_2006} and \cite{Faber_2009} with some modifications due to our setting. We have to mention that such arithmetic intersection theory has been discussed in \cite{chen2021arithmetic,moriwaki2016adelic} and \cite{Gubler_2003_5_2_4_711_0}. Their results show that the arithmetic intersection number is an integral (or a summation) of local intersection number. But here we consider the arithmetic intersection number as a limit of geometric intersection numbers, so that we can show the relationship between Arakelov geometry over function fields and relative geometry. 

We first need the following lemma about common modeling.
\begin{lemm}\label{lemm_comm_model} Let $L_1,\dots, L_n$
 be line bundles on $X$. Let \[(\mathscr{X}_1,\shfL_1),\dots,(\scrX_n,\shfL_n)\] be models of $(X,L_1),\dots,(X,L_n)$ respectively. Then there exists a $Y$-model $\scrX$ of $X$ and $Y$-morphisms $p_i:\scrX\rightarrow \scrX_i$ such that 
\begin{itemize}
    \item[\textnormal{(i)}] $\shfL_i':=p_i^*\shfL_i$ is a model of $L_i$.
    \item[\textnormal{(ii)}] For each place $\omega\in Y^{(1)}$, the induced formal metrics $\lvert\cdot\rvert_{\shfL_{i,\omega}}=\lvert\cdot\rvert_{\shfL'_{i,\omega}}$.
\end{itemize}
\end{lemm}
\begin{proof} See \cite[Lemma 2.2]{Faber_2009}.
\end{proof}

For $i\in \{0,\dots, d\}$, let $\overline L_i$ be a semipositive line bundles defined by the sequence $\{(\scrX_{i,n},\shfL_{i,n},l_{i,n})\}_{n\in\N_+}$.
For each $n$, we may assume that $\scrX_{0,n}=\cdots=\scrX_{d,n}=\scrX_n$ by Lemma \ref{lemm_comm_model}. We denote by $\pi_n$ the morphism $\scrX_n\rightarrow Y$. Then we define the intersection number by
$$\widehat c_1(\overline L_0)\cdots \widehat c_1(\overline L_d):=\limnto \frac{c_1(\shfL_{0,n})\cdots c_1(\shfL_{d,n})\cdot\pi_n^*{\mathcal{H}}}{l_{0,n}\cdots l_{d,n}}.$$
We are going to show that this intersection number is well-defined, which is ensured by the following estimate.
\begin{prop}\label{prop_compar_int}
Let $\pi:\scrX\rightarrow Y$ be a model of $X$. Let $\shfL_1,\dots,\shfL_d$ be line bundles on $\scrX$. Let $L_i:=\shfL_i|_K$ for $i=1,\dots,d$. Let $\shfL$ and $\shfL'$ be line bundles on $\scrX$, such that $\shfL_K=\shfL'_K=L$ for some line bundle $L$ on $X$. Then 
$$\begin{aligned}
\lvert c_1(\shfL^\vee\otimes \shfL')&c_1(\shfL_1)\cdots c_1(\shfL_d)\cdot\pi^*{\mathcal H}\rvert\leqslant\\
&\deg_{L_1,\dots,L_d}(X)\sum\limits_{\omega\in Y^{(1)}}\degH(\omega)d(\lvert\cdot\rvert_{\shfL,\omega},\lvert\cdot\rvert_{\shfL',\omega})
\end{aligned}$$
\end{prop}
\begin{proof}
Since the intersection number and distance functions are invariant under a pull-back through a finite surjective morphism of degree $1$, we can assume that $\scrX$ is normal. 
We may view $\shfL_0:=\shfL^\vee\otimes \shfL'$ as a model of $\O_X$. We take a rational section $s$ of $\shfL_0$ such that $s|_{X}=1$. Then $\mathrm{div}(s)$ have no horizontal part. 
Hence the decomposition defined in subsection \ref{sect_vert_div} is of form $\mathrm{div}(s)=Z_1+Z_2$.
Considering the special fiber $\scrX_\omega:=\scrX\times_{Y}\mathrm{Spec}(\O_{Y,\omega}/\mathfrak{m}_{Y,\omega})$ over $\omega$, we have the reduction map $X_\omega^{\mathrm{an}}\rightarrow \scrX_\omega$ which is surjective and anticontinous. Moreover, for any irreducible component $W$ of $\scrX_\omega$ with multiplicity $\lambda_W$, due to the normality, its generic point $\eta_W$ corresponds to a unique point $\xi_W\in X^{\mathrm{an}}_\omega$ satisfying the following
$$-\ln\lvert 1\rvert_{\shfL_0,\omega}(\xi_W)=\mathrm{ord}_W(s)/\lambda_W.$$
Let $\alpha:=c_1(\shfL_1)\cdots c_1(\shfL_d)$. Then
\begingroup
\allowdisplaybreaks
\begin{align*}
\left\lvert\sum\limits_{W}\mathrm{ord}_W(s)\degH(\pi_*(\alpha\cdot[W]))\right\rvert&=\left\lvert\sum\limits_{W}-\ln\lvert 1\rvert_{\shfL_0,\omega}(\xi_W)\lambda_W\degH(\pi_*(\alpha\cdot[W]))\right\rvert\\
&\leqslant\max\limits_{x\in X^{\mathrm{an}}_{\omega}}{\lvert 1\rvert_{\shfL_0,\omega}}(x)\lvert\degH(\pi_*(\alpha\cdot [\sum \limits_{W}\lambda_{W}W]))\rvert\\
&=\max\limits_{x\in X^{\mathrm{an}}_{\omega}}{\lvert 1\rvert_{\shfL_0,\omega}}(x)\lvert\degH(\pi_*(\alpha\cdot[\scrX_\omega]))\rvert\\
&=\max\limits_{x\in X^{\mathrm{an}}_{\omega}}{\lvert 1\rvert_{\shfL_0,\omega}(x)}\degH(\omega)\rvert \mathrm{deg}_{L_1\cdots L_d}(X)\rvert
\end{align*}
\endgroup
where the last equality is due to subsection \ref{sect_alg_cyc}. Take a summation of above equation on $\omega\in Y^{(1)}$, and combine with the fact that $\degH(\pi_*(\mathrm{div}(s)\cdot\alpha))=\degH(\pi_*([Z_1]\cdot \alpha))$, we conclude the proof.
\end{proof}

For any positive integers $n$ and $m$, we may assume that $\scrX_n=\scrX_m=\scrX$ and denote by $\pi$ the morphism $\scrX\rightarrow Y$.
\begingroup
\allowdisplaybreaks
\begin{align*}
&\left\lvert \frac {c_1(\shfL_{0,n}\cdots\shfL_{d,n})\cdot\pi^*\mathcal H}{l_{0,n}\cdots l_{d,n}}-\frac{c_1(\shfL_{0,m}\cdots\shfL_{d,m})\cdot\pi^*\mathcal H}{l_{0,m}\cdots l_{d,m}}\right\rvert\\
    &=\frac{\left\lvert c_1(\shfL_{0,n}^{\ot l_{0,m}})\cdots c_1(\shfL_{d,n}^{\ot l_{d,m}})\cdot\pi^*\mathcal H- c_1(\shfL_{0,m}^{\ot l_{0,n}})\cdots c_1(\shfL_{d,m}^{\ot l_{d,n}})\cdot\pi^*\mathcal H\right\rvert}{l_{0,n}\cdots l_{d,n} l_{0,m}\cdots l_{d,m}}\\
    &\leqslant\frac{\sum\limits_{0\leqslant i\leqslant d}\left\lvert c_1(\shfL_{0,n}^{\ot l_{0,m}})\cdots c_1(\shfL_{i,n}^{\ot l_{i,m}}\ot(\shfL_{i,m}^{\vee})^{l_{i,n}})\cdots c_1(\shfL_{d,m}^{\ot l_{d,n}})\cdot\pi^*\mathcal H\right\rvert}{l_{0,n}\cdots l_{d,n} l_{0,m}\cdots l_{d,m}}\\
    &\leqslant \sum\limits_{0\leqslant i\leqslant d}\left(\deg_{L_0\cdots L_{i-1}L_{i+1}\cdots L_d}(X)\sum_{\omega\in Y^{(1)}}\degH(\omega)d\left(\lvert\cdot\rvert_{\shfL_{i,n}}^{1/l_{i,n}},\lvert\cdot\rvert_{\shfL_{i,m}}^{1/l_{i,m}}\right)\right).
\end{align*}
\endgroup
Therefore the sequence $\displaystyle{\left\{\frac{c_1(\shfL_{0,n})\cdots c_1(\shfL_{d,n})\cdot\pi_n^*{\mathcal H}}{l_{0,n}\cdots l_{d,n}}\right\}}$ in the definition is a Cauchy sequence.
In the same manner we can prove that the definition is independent of the choice of the defining sequence, so our definition is well-defined.
In general, we consider a closed subvariety $Z\subset X$ of dimension $d'\leqslant d$. Then for any semipositive line bundle $\overline L=(L,\{\lvert\cdot\rvert_{\omega}\})$ with the defining sequence $\{(\scrX_n,\shfL_n, l_n)\}$, the restriction $\overline L|_Z:=(L|_Z,\{\lvert\cdot\rvert_\omega|_{Z_\omega^\mathrm{an}}\})$ is a semipositive line bundle on $Z$ defined by $\{(\mathscr{Z}_n,\shfL_n|_{\mathscr{Z}_n},l_n)\}$ where each $\mathscr{Z}_n$ is the closure of $Z$ in $\mathscr{X}_n$. Moreover, we consider semipositive line bundles $\overline L_0,\dots,\overline L_{d'}$, and a purely $d'$-dimensional cycle $\alpha=\sum n_i Z_i$. The intersection number is defined by
$$\widehat c_1(\overline L_0)\cdots\widehat c_1(\overline L_{d'})\cdot \alpha:=\sum n_i \widehat c_1(\overline L_0|_{Z_i})\cdots\widehat c_1(\overline L_{d'}|_{Z_i})$$
We can easily extend this definition to integrable line bundles by linearity.
Let $\overline L$ be a semipositive line bundles. Then we define the \textit{height} of $\alpha$ with respect to $\overline L$ by
$$h_{\overline L}(\alpha)=\frac{\widehat c_1(\overline L)^{d'+1}\cdot [\alpha]}{\deg_{L}(\alpha)(d'+1)},$$
where $\deg_L(\alpha)=c_1(L)^{d'}\cdot\alpha$. 

A frequently used result about the relationship between heights and asymptotic minimal slopes is the following:
\begin{prop}\label{prop_min_slop_nef}
Let $\scrX\rightarrow Y$ be a projective and surjective morphism of normal and projective $k$-varieties. Let $\shfL$ be a nef line bundle over $\scrX$ with $\shfL_K$ being ample. Then we have
 $$\mu^{\mathrm{asy}}_{\min}(\shfL)\geqslant 0$$
\end{prop}
\begin{proof}
Let $\ovl L$ be the adelic line bundle $(\shfL_K,\{\lvert\cdot\rvert_{\shfL,\omega}\})$ associated to $\shfL$. Since $\shfL$ is nef, 
by Proposition \ref{prop_min_slop_inv} and \cite[Theorem 5.2.23]{chen2022hilbert}, we conclude the proof.
\end{proof}

\subsection{Chambert-Loir measures}
\begin{defi}
Let $\phi$ be a model metric of $\O_{X_\omega}$. Let $l$ be a positive integer. We can define a continuous function $f=-\ln\lvert 1\rvert_{\phi}^{1/l}$. Such a function is called a \textit{model function}.
If $l=1$, we denote by $\overline{O(f)}$ the adelic line bundle $(\O_X,\{\lvert\cdot\rvert_{f,\omega'}\}_{\omega'\in Y^{(1)}})$ with $$\lvert\cdot\rvert_{f,\omega'}=\begin{cases}\phi,\text{ if }\omega'=\omega,\\
\text{trivial elsewhere.}\end{cases}$$
\end{defi}
 Due to \cite[Lemma 3.5]{Yuan_2008}, for any formal metric $\phi$ of $\shO_{X,\omega}^{\mathrm{an}}$, there exists a $Y$-model $(\scrX,\O(f))$ of $(X,\O_{X})$, such that $\O(f)$ induces the adelic line bundle $\overline{O(f)}$. For example, take arbitrary model $\pi_0:\scrX_0\rightarrow Y$ of $X$, let $f:\scrX\rightarrow \scrX_0$ be the blowing-up of $\scrX_0$ along $\pi_0^{-1}(\overline{\{\omega\}})$. Let $\pi:=\pi_0\circ f$. Then  $(\scrX,\O_{\scrX}(\pi^{-1}(\overline{\{\omega\}}))$ induces the constant function $f=1$ on $X_\omega^{\mathrm{an}}$. Let $\overline L$ be a semipositive adelic line bundle on $X$, and $Z$ be a $d'$-dimensional closed subvariety of $X$.
By the \textit{Chambert-Loir measure} $\mu_{\overline L,Z,\omega}$ we mean a probability measure such that for any model function $f$,
$$\int_{X_{\omega}^{\mathrm{an}}}f\mu_{\overline L,Z,\omega}=\frac{\widehat c_1(\overline L)^{d'}\cdot\widehat{c}_1(\overline{O(f)})\cdot[Z]}{\degH(\omega)\deg_L(Z)} $$
since model functions are uniformly dense in the space of continuous functions on $X_{\omega}^{\mathrm{an}}$ \cite[Theorem 7.12]{gubler1998local}. In particular, if $Z=X$, we write $\mu_{\overline L,\omega}:=\mu_{\overline L,X,\omega}$.
\begin{rema}
Let $\rho:X'\rightarrow X$ be a generically finite surjective morphism. We can compute that $$\frac{\widehat c_1(\rho^*\overline L)^d \widehat c_1(\overline{O(\rho^*f)})}{\degH(\omega)\deg_{\rho^*L}(X')}=\frac{\deg(\rho)\widehat c_1(\overline L)^d\widehat c_1(\overline{O(f)})}{\deg(\rho)\degH(\omega)\deg_L(X)}=\frac{\widehat c_1(\overline L)^d\widehat c_1(\overline{O(f)})}{\degH(\omega)\deg_L(X)}$$
holds for any model function $f:X^{\mathrm{an}}_\omega\rightarrow \R$. Therefore $\mu_{\overline L,\omega}=\rho_*\mu_{\rho^*\overline L,\omega}$.
\end{rema}

\section{A relative bigness inequality}\label{sec_rel_big}
Let $\pi:\scrX\rightarrow Y$ be a surjective and projective morphism of projective $k$-varieties where $\scrX$ is of dimension $d+e$ ($d \in \N$). 
\subsection{Relative version of Siu's inequality}
We begin with an easy application of Riemann-Roch theorem.
\begin{lemm}[Relative asymptotic Riemann-Roch theorem]\label{lemm_asy_rel_RR}
Let $\shfF$ be a coherent sheaf on $\scrX$. Let $\shfL$ be a relatively ample line bundle on $\scrX$. Then
\begin{equation}\label{eq_rel_asy_RR}
    \degH(\pi_*(\shfF\ot \shfL^{\ot n}))=\rank(\shfF)\frac{c_1(\shfL)^{d+1}\cdot\pi^*\mathcal H}{(d+1)!}n^{d+1}+O(n^{d})
\end{equation}
\end{lemm}
\begin{proof}
Let $U$ be the maximal Zariski open set such that $U$ is regular and $\rest{\scrX}{U} \to U$ is flat. 
Note that $\operatorname{codim} (Y \setminus U) \geqslant 2$ and \[R^i\pi_*(\shfF\ot\shfL^{\ot n}|_{\pi^{-1}(U)})\simeq R^i\pi_*(\shfF\ot\shfL^{\ot n})|_U=0\]%
for every $i>0$ and $n\gg 0$. Therefore For $n\gg 0$, Grothendieck-Riemann-Roch theorem leads to
$$\label{eq_GRR}
    \ch(\pi_*(\shfF\ot \shfL^{\ot n})|_U)\cap\Td(U)=\pi_*(\ch(\shfF\ot \shfL^{\ot n}|_{\pi^{-1}}(U))\cap\Td(\pi^{-1}(U))).
$$
We only consider the part of above equation in $A_{e-1}(U)_\Q$, which is
$$c_1(\pi_*(\shfF\ot \shfL^{\ot^n})|_U)=\rank(\shfF)\frac{\pi_*c_1(\shfL|_{\pi^{-1}(U)})^{d+1}}{(d+1)!}n^{d+1}+P(n),$$
where $P(n)$ is a polynomial of degree $d$ with coefficients in $A_{e-1}(U)_\Q$. The flatness gives that $c_1(\shfL|_{\pi^{-1}(U)})^{d+1}=c_1(\shfL)^{d+1}|_{\pi^{-1}(U)}$. By virtue of the fact that $\pi_*(c_1(\shfL)^{d+1}|_{\pi^{-1}(U)})=\pi_* c_1(\shfL)^{d+1}$, we obtain that $$c_1(\pi_*(\shfF\ot \shfL^{\ot^n}))=\rank(\shfF)\frac{\pi_*c_1(\shfL)^{d+1}}{(d+1)!}n^{d+1}+P(n).$$
After taking the intersection multiplicities of both sides with $c_1(H_1)\cdots c_1(H_{e-1})$, we obtain (\ref{eq_rel_asy_RR}) by projection formula.
\end{proof}

\begin{lemm}\label{lemm_deg_birational}
Let $f:\scrX'\rightarrow \scrX$ be a birational proper morphism. Let $\mathscr E$ and $\shfL$ be line bundles on $\scrX$. 
Let $\pi'$ denote the surjective morphism $\pi\circ f$. 
We have $$\degH(\pi_*(\mathscr E\ot \shfL^{\ot n}))\geq \degH(\pi'_*(f^*\mathscr E\ot (f^*\shfL)^{\ot n}))+O(n^{d})$$
\end{lemm}
\begin{proof}
We need to give an upper bound of the following
{
\allowdisplaybreaks
\begin{align*}
&\degH(\pi'_*(f^*(\mathscr{E}\ot\shfL^{\ot n})))-\degH(\pi_*(\mathscr{E}\ot\shfL^{\ot n}))\\
&=\degH(\pi_*(\mathscr{E}\ot\shfL^{\ot n}\ot f_*\O_{\scrX'}))-\degH(\pi_*(\mathscr{E}\ot\shfL^{\ot n}))\\
&=\degH(\pi_*(\mathscr{E}\ot\shfL^{\ot n}\ot f_*\O_{\scrX'})/\pi_*(\mathscr{E}\ot\shfL^{\ot n})).
\end{align*}
}%
The first equation is due to projection formula.
Let $\scrA$ be an ample line bundle on $\scrX$ such that $\shfL\ot\scrA$ is ample, $\mu_{\min}^{\mathrm{asy}}(\pi_*(\shfL\ot\scrA))>0$, and $\scrA$ admits a global section $t$ not vanishing at any $x\in \mathrm{Supp}(f_*\O_{\scrX'}/\O_{\scrX})$. Then we have an injection 
$$
\begin{aligned}
\pi_*(\mathscr{E}\ot\shfL^{\ot n}\ot f_*\O_{\scrX'})&/\pi_*(\mathscr{E}\ot\shfL^{\ot n})\xrightarrow{\cdot t^{\ot n}}\\ &\pi_*(\mathscr{E}\ot(\shfL\ot\scrA)^{\ot n}\ot f_*\O_{X'})/\pi_*(\mathscr{E}\ot(\shfL\ot\scrA)^{\ot n})
\end{aligned}$$
Denote by $\shfF_1(n)$ the former one, and by $\shfF_2(n)$ the latter one.
We claim that $\degH(\shfF_2(n))\geqslant \degH(\shfF_1(n))$.
If we let $T(n)$ be the torsion part of $\shfF_2(n)$, then $T(n)\cap \shfF_1(n)$ is the torsion part of $\shfF_1(n)$. If $\shfF_2(n)/T(n)$ is zero, then $\shfF_1(n)$ and $\shfF_2(n)$ are both torsion, hence $\degH(\shfF_2(n))\geqslant \degH(\shfF_1(n))$ due to Remark \ref{rema_pos_tor}. Therefore it suffices to consider the case that $\shfF_2(n)/T(n)$ is torsion-free. 
By our assumption on minimal slope, $\mu_{\min}(\shfF_2(n)/T(n))>0$ for every $n\gg0$. We thus have 
{
\allowdisplaybreaks
\begin{align*}
\degH(&\shfF_2(n))-\degH(\shfF_1(n))\\
&=\degH(\shfF_2(n)/T(n))-\degH(\shfF_1(n)/(T(n)\cap\shfF_1(n)))\\
&+\degH(T(n))-\degH(T(n)\cap \shfF_1(n))
\end{align*}
}%
which is non-negative, as required. In consequence,
{
\allowdisplaybreaks
\begin{align*}
&\degH(\pi'_*(f^*(\mathscr{E}\ot\shfL^{\ot n})))-\degH(\pi_*(\mathscr{E}\ot\shfL^{\ot n}))\\
&\leqslant\degH(\pi_*(\mathscr{E}\ot(\shfL\ot\scrA)^{\ot n}\ot f_*\O_{X'}))-\degH(\pi_*(\mathscr{E}\ot(\shfL\ot\scrA)^{\ot n}))\\
&=O(n^d).
\end{align*}}%
The last equation is obtained by using Lemma \ref{lemm_asy_rel_RR}
\end{proof}

\begin{theo}\label{ineq_rel_big} Assume that $d>0$.
Let $\mathscr{E}, \shfL, \shfM$ be line bundles over $\scrX$. If $\shfL$ and $\shfM$ are nef,
then it holds that
 \begin{equation}\label{ineq_main_geo}
    \begin{aligned}
        \degH(&\pi_*(\mathscr{E}\otimes (\shfL\ot \shfM^\vee)^{\ot n}))\geqslant \\
             &\frac{c_1(\shfL)^{d+1}\cdot\pi^*(\mathcal{H})-(d+1)c_1(\shfL)^{d}c_1(\shfM)\cdot\pi^*(\mathcal{H})}{(d+1)!}n^{d+1}+o(n^{d+1}).
    \end{aligned}
 \end{equation}
\end{theo}
\begin{proof}
We will divide the proof into three steps.

\medskip
\textbf{Step 1.} We first show that we can reduce to the case that 
\begin{enumerate}
    \item[\textnormal{(i)}] $\scrX$ is smooth.
    \item[\textnormal{(ii)}] $\shfL$ and $\shfM$ are ample.
    \item[\textnormal{(iii)}] $\mu^{\mathrm{asy}}_{\min}(\shfL)>0$.
\end{enumerate}

Let $f:\scrX'\rightarrow \scrX$ be an resolution of singularities, and $\pi'=\pi\circ f$. Let $\mathscr E'=f^* \mathscr E$, $\shfL'=f^*\shfL$ and $\shfM'=f^*\shfM$.

We claim that we can find an ample line bundle $\scrA$ over $\scrX'$ such that $\mu_{\min}^{\mathrm{asy}}(\shfL'^{\ot m}\ot \scrA)>0$. 
Indeed, let $\scrA_0$ be an ample line bundle over $\scrX'$.
Then, 
$\mu_{\min}^{\mathrm{asy}}(\shfL'^{\ot m}\ot \scrA_0)\geqslant 0$ by Proposition~\ref{prop_min_slop_nef}.
If we set $\scrA = \scrA_0\ot \pi'^*A$ for an ample line bundle $A$ over $Y$, then $\scrA$ is our desired line bundle by $$\mu_{\min}^{\mathrm{asy}}(\shfL'\ot \scrA)=\mu_{\min}^{\mathrm{asy}}(\shfL'\ot \scrA_0)+\degH(A)>0.$$

We denote the inequality \eqref{ineq_main_geo} by $I(\mathscr{E}, \shfL, \shfM)$.
We are going to see that
if \[I(\mathscr E'\ot (\shfL'\ot \shfM'^{\vee})^{\ot r},\shfL'^{\ot m}\ot \scrA,\shfM'^{\ot m}\ot \scrA)\] holds for all
$m\in \N_+, 0\leqslant r<m$, then so does $I(\mathscr{E}, \shfL, \shfM)$.

Set $\shfL_m=\shfL'^{\ot m}\ot \scrA$, $\shfM_m=\shfM'^{\ot m}\ot \scrA$, $\mathscr N=\shfL'\ot \shfM'^{\vee}$ and \[
P(N) =  \degH(\pi'_*(\mathscr E'\ot \mathscr{N}^{\ot N})).\]
Then $I(\mathscr E'\ot (\shfL'\ot \shfM'^{\vee})^{\ot r},\shfL'^{\ot m}\ot \scrA,\shfM'^{\ot m}\ot \scrA)$ means that
\[
\liminf_{n\to\infty} \frac{P(nm+r)}{n^{d+1}} \geqslant 
\frac{(c_1(\shfL_m)^{d+1}-(d+1)c_1(\shfL_m)^{d}c_1(\shfM_m)){\pi'}^*\mathcal H}{(d+1)!},
\]
that is,
\[
\liminf_{n\to\infty} \frac{P(nm+r)}{(nm+r)^{d+1}} \geqslant 
\frac{(c_1(\shfL_m)^{d+1}-(d+1)c_1(\shfL_m)^{d}c_1(\shfM_m)){\pi'}^*\mathcal H}{m^{d+1}(d+1)!}.
\]

Note that, for a sequence $\{ a_n \}_{n=1}^{\infty}$ of real numbers,
\[
\min_{r \in \{ 0, \ldots, m-1 \}} \left\{ \liminf_{n\to\infty} a_{nm+ r} \right\} = \liminf_{n\to\infty} a_n,
\]
so we have
\[    \liminf_{n\to\infty}\frac{P(n)}{n^{d+1}} \geqslant \frac{(c_1(\shfL_m)^{d+1}-(d+1)c_1(\shfL_m)^{d}c_1(\shfM_m)){\pi'}^*\mathcal H}{m^{d+1}(d+1)!}.
\]
Therefore,
\[    \liminf_{n\to\infty}\frac{P(n)}{n^{d+1}} \geqslant \liminf_{m\to\infty} \frac{(c_1(\shfL_m)^{d+1}-(d+1)c_1(\shfL_m)^{d}c_1(\shfM_m)){\pi'}^*\mathcal H}{m^{d+1}(d+1)!}.
\]
Thus, it is sufficient show that
\[
\liminf_{n\to\infty}\frac{P(n)}{n^{d+1}} \leqslant \liminf_{n\to\infty}\frac{\degH(\pi_*(\mathscr{E}\otimes (\shfL\ot \shfM^\vee)^{\ot n}))}{n^{d+1}}
\]
and 
\begin{multline*}
 \lim_{m\to\infty} \frac{(c_1(\shfL_m)^{d+1}-(d+1)c_1(\shfL_m)^{d}c_1(\shfM_m)){\pi'}^*\mathcal H}{m^{d+1}} \\
= (c_1(\shfL)^{d+1}-(d+1)c_1(\shfL)^{d}c_1(\shfM))\cdot\pi^*(\mathcal{H}).
\end{multline*}
The first inequality is due to Lemma \ref{lemm_deg_birational}, and the second assertion is obvious by virtue of the projection formula.

\medskip
\textbf{Step 2.} In this step, we prove that if there is an effective section $s$ of $\shfM$ such that $D:=\mathrm{div}(s)$ is 
a prime divisor, then as $n\rightarrow+\infty,$
\begin{equation}
    \begin{aligned}\label{eq_estimate_diff}
    \degH(\pi_*(\mathscr E\ot\shfL^{\ot n}\ot (\shfM^\vee)^{\ot j-1}))-\degH(\pi_*(\mathscr E\ot\shfL^{\ot n}\ot (\shfM^\vee)^{\ot j}))\\
    \leqslant\frac{c_1(\shfL)^d\cdot D\cdot\pi^*(\mathcal{H})}{d!}n^d+O(n^{d-1})
    \end{aligned}
\end{equation}
where $0\leqslant j\leqslant n$ and $O(n^{d-1})$ is independent of $j$.

Consider the exact sequence:
\begin{align*}
    &0\rightarrow \mathscr E\ot \shfL^{\ot n}\ot(\shfM^\vee)^{\ot j}\xrightarrow{\cdot s}\mathscr E\ot \shfL^{\ot n}\ot(\shfM^\vee)^{\ot j-1}\\ &\kern15em\to (\mathscr E\ot \shfL^{\ot n}\ot (\shfM^\vee)^{\ot j-1})|_D \to 0,
\end{align*}
which induces 
\begin{align*}
    0\to \pi_*(\mathscr E\ot \shfL^{\ot n}\ot (\shfM^\vee)^{\ot j})\xrightarrow{\pi_*(\cdot s)}\pi_*(\mathscr E\ot \shfL^{\ot n}\ot (\shfM^\vee)^{\ot j-1})\\
    \to\pi_*((\mathscr E\ot \shfL^{\ot n}\ot (\shfM^\vee)^{\ot j-1})|_D)\to\cdots.
\end{align*}
If we set $\shfF:=\mathrm{Coker}(\pi_*(\cdot s))$, then
\begin{align*}
    \degH(\pi_*(\mathscr E\ot\shfL^{\ot n}\ot (\shfM^\vee)^{\ot j-1})) -\degH(\pi_*(\mathscr E\ot\shfL^{\ot n}\ot (\shfM^\vee)^{\ot j}))\\ = \degH(\shfF),
\end{align*}
so that we are going to give an upper bound of $\degH(\shfF)$.
Consider the composition of the following two injective morphisms:
$$
\begin{cases}\shfF \rightarrow \pi_*((\mathscr E\ot \shfL^{\ot n}\ot (\shfM^\vee)^{\ot j-1})|_{D}),\\
\pi_*((\mathscr E\ot \shfL^{\ot n}\ot (\shfM^\vee)^{\ot j-1})|_{D})\rightarrow \pi_*(\mathscr E\ot\shfL^{\ot n}|_{D}).
\end{cases}
$$
Since $D$ is ample, we have $\pi(D)=Y$, which
gives that $\shfF$ and $\pi_*(\mathscr E\ot \shfL^{\ot n}|_{D})$ are torsion-free or zero. Hence  $$\begin{aligned}
    \liminf_{n\rightarrow +\infty}\frac{\mu_{\min}(\pi_*(\mathscr E\ot \shfL^{\ot n}|_{D}))}{n}=\liminf_{n\rightarrow +\infty}\frac{
    \mu_{\min}(\pi_*(\mathscr E\ot \shfL^{\ot n}\ot \O_{D}))}{n}\\
    \geqslant \mu_{\min}^{\mathrm{asy}}(\shfL)>0,
\end{aligned}$$
where the first inequality is due to Proposition \ref{prop_asymp_min_slop}(b), and the second inequality is our reduction in Step 1.
Therefore for every sufficiently large $n$,
$$\mu_{\min}(\pi_*(\mathscr E\ot \shfL^{\ot n}|_D))>0.$$
We see that
\begingroup
\allowdisplaybreaks
    \begin{align*}
    \degH(\shfF)&\leqslant\degH(\pi_*(\mathscr E\ot \shfL^{\ot n}|_{D}))
    =\frac{c_1(\shfL|_{D})^d\cdot \pi^*\mathcal H}{d!}n^d+O(n^{d-1})\\
    &=\frac{c_1(\shfL)^d \cdot D\cdot\pi^*(\mathcal{H})}{d!}n^d+O(n^{d-1})
\end{align*}
\endgroup
as $n\rightarrow +\infty$ due to Lemma \ref{lemm_asy_rel_RR}.

\medskip
\textbf{Step 3.} Now we prove the inequality in the general case. As we have seen in Step 1, it suffices to prove that there exists an $m>0$, such that the inequality $I(\mathscr E\ot (\shfL\ot \shfM^{\vee})^{\ot r},\shfL^{\ot m},\shfM^{\ot m})$ holds for any $0\leqslant r<m$.
After taking multiples of $\shfL$ and $\shfM$ simultaneously, there is no loss of generality to assume that $\shfM$ is very ample, and hence there exists a non-zero global section $s$ such that $D= \mathrm{div}(s)$ is irreducible and reduced by
the Bertini irreducible theorems because of the fact $d+e \geqslant 2$.
Note that the infinitesimal $O(n^{d-1})$ in (\ref{eq_estimate_diff}) is independent of the choice of $j$.
Taking a summation of (\ref{eq_estimate_diff}) on $j$ and applying Lemma \ref{lemm_asy_rel_RR}, we obtain that
$$\begin{aligned}
    &\degH(\pi_*(\mathscr E\ot \shfL^{\ot n}\ot (\shfM^\vee)^{\ot n}))=\degH(\pi_*(\mathscr E\ot \shfL^{\ot n}))\\
    &+\sum\limits_{1\leqslant j\leqslant n}\{\degH(\pi_*(\mathscr E\ot \shfL^{\ot n}\ot (\shfM^\vee)^{\ot j}))-\degH(\pi_*(\mathscr E\ot \shfL^{\ot n}\ot (\shfM^\vee)^{\ot j-1})\}\\
    &\kern4em\geqslant \frac{c_1(\shfL)^{d+1}\cdot\pi^*(\mathcal{H})-(d+1)c_1(\shfL)^{d}c_1(\shfM)\cdot\pi^*(\mathcal{H})}{(d+1)!}n^{d+1}+o(n^{d+1})& 
\end{aligned}$$
\end{proof}

Here we give some results related to our latter contents about equidistribution theorem.

\begin{coro}\label{coro_ineq_err}
 Let $\shfL,\shfM_1,\shfM_2$ be nef line bundles on $\scrX$ such that $(\shfM_2)_K=(\shfM_1)_K$.
 Then there exists an function $F:\N\rightarrow \R$ such that $F(m)=O(m^{d-1})$ as $m\rightarrow +\infty$, and that for each fixed $m$, as $n\rightarrow \infty$, we have
 $$\begin{aligned}
&\degH(\pi_*(\shfL^m\ot\shfM_1\ot\shfM_2^\vee)^n)\\
&\geqslant \frac{c_1(\shfL^m\ot \shfM_1\ot \shfM_2^\vee)^{d+1}\cdot\pi^*\mathcal H+F(m)}{(d+1)!}n^{d+1}+o(n^{d+1}).
\end{aligned}$$
\end{coro}
\begin{proof}
For any fixed $m\gg 0$,
applying the Theorem \ref{ineq_rel_big} to $\shfL^{\ot m}\ot \shfM_1$ and $\shfM_2$ yields the inequality
$$\begin{aligned}
&\degH(\pi_*(\shfL^{\ot m}\ot \shfM_1\ot \shfM_2^\vee)^n)\geqslant \\
&\frac{(c_1(\shfL^{\ot m}\ot \shfM_1)^{d+1}-(d+1)c_1(\shfL^{\ot m}\ot \shfM_1)^dc_1(\shfM_2))\cdot\pi^*\mathcal{H}}{(d+1)!}n^{d+1}+o(n^{d+1}).
\end{aligned}$$
Let $$\begin{aligned}
F(m):=&(c_1(\shfL^{\ot m}\ot \shfM_1)^{d+1}-(d+1)c_1(\shfL^{\ot m}\ot\shfM_1)^dc_1(\shfM_2)\\
&-c_1(\shfL^{\ot m}\ot \shfM_1\ot\shfM_2^\vee)^{d+1})\cdot \pi^*\mathcal{H}.
\end{aligned}$$
Then $F(m)\sim O(m^{d-1})$ as $m\rightarrow +\infty$.
\end{proof}
\begin{rema}\label{rema_err_term}
Note that $F(m)$ above depends on $\shfL$, $\shfM_1$ and $\shfM_2$. Then we have an estimate that
$$F(m)\leqslant 3^{d+1}\max\limits_{a+b+c=d+1}\{c_1(\shfL)^a c_1(\shfM_1)^b c_1(\shfM_2)^c\cdot\pi^*\mathcal{H}\}m^{d-1}.$$
So in the following content, we denote that
$$C(\shfL,\shfM_1,\shfM_2):=3^{d+1}\max\limits_{a+b+c=d+1}\{c_1(\shfL)^a c_1(\shfM_1)^b c_1(\shfM_2)^c\cdot\pi^*\mathcal{H}\}.$$
\end{rema}

\section{Equidistribution over function fields}\label{sec_equi}
In this section, we assume that $Y$ is a projective regular curve over $k$. In this case, we do not need a polarization $\mathcal H$. For each closed point $\omega\in Y$, we simply denote by $\deg(\omega)$ the degree $[\kappa(\omega):k]$.
\subsection{Equidistribution of subvarieties}
Let $X$ be an integral projective variety over the function field $K$. Let $I$ be an infinite directed set. A net $\{Z_{\iota}\}_{\iota\in I}$ of subvarieties is a family of closed subvarieties $Z_\iota\subset X$ indexed by $I$. We say $\{Z_\iota\}$ is \textit{generic} if for any proper closed subset $Z\subset X$, there exists an $\iota_0$, such that $Z_\iota \not\subset Z$ for any $\iota>\iota_0$.
\begin{defi}
Let $\ovl L$ be a semipositive adelic line bundle. We say a generic net $\{Z_\iota\}_{\iota \in I}$ of closed subvarieties is equidistributed with respect to $\ovl L$ if for any $\omega\in Y^{(1)}$, the family of probability measures $\{\mu_{\ovl L,Z_\iota,\omega}\}_{\iota\in I}$ indexed by $I$ weakly converges to $\mu_{\ovl L,\omega},$ that is,
\begin{equation}\label{eq_equidistribution}
    \int_{X_\omega^{\mathrm{an}}}f \mu_{\overline L,\omega}=\lim\limits_{\iota\in I} \int_{X_\omega^{\mathrm{an}}}f \mu_{\overline L,Z_\iota,\omega}
\end{equation}
holds for any continuous function $f:X^{\mathrm{an}}_{\omega}\rightarrow \R$.
\end{defi}
\begin{theo}\label{theo_equidistribution}
Let $\overline L$ be an adelic line bundle such that $L$ is ample and $\widehat{c}_1(\overline L)^{d+1}=0$. Moreover, there exists a sequence of triples \[\{(\pi_i:\scrX_i\rightarrow Y,\shfL_i,l_i)\}_{i\in \N_+}\] such that 
\begin{enumerate}
    \item[\textnormal{(a)}] Each $(\scrX_i,\shfL_i)$ is a $Y$-model of $(X,L^{\ot l_i})$, and $\shfL_i$ is nef.
    \item[\textnormal{(b)}] Let $\ovl L_i:=(L,\{\lvert\cdot\rvert_{\shfL_i,\omega}^{1/l_i}\}).$ Then $\lim\limits_{n\rightarrow+\infty} d(\ovl L_i,\ovl L)=0.$
\end{enumerate}
Let $\{Z_\iota\}_{\iota\in I}$ be a generic net of subvarieties of $X$ such that  $\lim\limits_{\iota\in I} h_{\overline L}(Z_\iota)=0.$ Then $\{Z_\iota\}$ is equidistributed with respect to $\ovl L$.
\end{theo}
\begin{proof}
Let $\rho:X'\rightarrow X$ be the normalization of $X$. For any $\iota\in I$, the base change $\rho^{-1}(Z_\iota)\rightarrow Z_\iota$ is finite. We take an irreducible component $Z'_\iota$ of $\pi^{-1}(Z_\iota)$. Then $\rho_*\mu_{\rho^*\overline L,Z'_\iota,\omega}=\mu_{\overline L,Z_\iota,\omega}$ due to the finiteness. It is easy to see that the net $\{Z'_\iota\}_{\iota\in I}$ of subvarieties of $X'$ is also generic, and $\lim\limits_{\iota\in I}h_{\rho^*\overline L}(Z_\iota')=0$. We may thus assume $X$ is normal.

We claim that if $\displaystyle \int_{X_\omega^{\mathrm{an}}}f d\mu_{\overline L,\omega}>0$, then 
$\displaystyle\liminf\limits_{\iota\in I}\int_{X_\omega^{\mathrm{an}}}f\mu_{\overline L,Z_\iota,\omega}\geqslant 0.$

Since the set of model functions is uniformly dense in the space of continuous functions on $X_\omega^{\mathrm{an}}$, we only need to show the claim holds for all model functions. By linearity, we assume that $f$ is a model function given by a $Y$-model $(\scrX,\mathscr{O}(f))$ of $(X,\O_X)$. Consider a decomposition $\O(f)=\shfM_1\ot\shfM_2^\vee$ by ample line bundles $\shfM_1$ and $\shfM_2$. Let $\overline \shfM_1$ and $\overline \shfM_2$ denote the associated adelic line bundles. 

For each $i>0$, by the simultaneous modeling lemma, we may consider $\shfM_1$ and $\shfM_2$ as nef line bundles over $\scrX_n$. 
Now we are going to give an estimate of $C(\shfL_i,\shfM_1,\shfM_2)$ defined in Remark \ref{rema_err_term}.
We set \[
\begin{cases}
C(\overline L):=3^{d+1}\max\limits_{\substack{a+b+c=d+1\\ a,b,c\in\N}}\{\widehat c_1(\overline L)^{a}\cdot\widehat c_1(\overline M_1)^b \cdot\widehat c_1(\overline M_2)^c\},\\
C_L:=3^{d+1}\max\limits_{\substack{a+b+c=d+1\\ a,b,c\in\N}}\{c_1(L)^a c_1(M_1)^b c_1(M_2)^c\}.\\
\end{cases}
\]
For $a,b,c\in \N$ such that $a>0$ and $a+b+c=d+1$, 
Proposition \ref{prop_compar_int} gives  $$\begin{aligned}
c_1(\shfL_i)^a\cdot c_1(\shfM_1^{\ot l_i})^b\cdot c_1(\shfM_2^{\ot l_i})^c-l_i^{d+1}\widehat c_1(\overline L)^a\cdot \widehat c_1(\overline M_1)^b\cdot\widehat c_1(\overline M_2)^c\\
\leqslant al_i^{d+1}d(\overline L,\overline L_i)c_1(L)^{a-1}\cdot c_1(M_1)^b\cdot c_1(M_2)^c.
\end{aligned}$$
Then we can see that 
$$C(\shfL_i, \shfM_1^{\ot l_i},\shfM_2^{\ot l_i})
\leqslant l_i^{d+1}(C(\overline L)+(d+1)d(\overline L,\overline L_i))C_L).$$
Therefore there exists a constant $C$ depending on $\overline L$, $\overline M_1$ and $\overline M_2$ such that $$C(\shfL_i,\shfM_1^{\ot l_i},\shfM_2^{\ot l_i})\leqslant l_i^{d+1}C.$$

Observe that
\begin{align*}
    \widehat c_1(\overline L^{\ot m}\otimes \overline{O(f)})^{d+1}&=m^d(d+1)\widehat{c}_1(\overline L)^{d}\widehat{c}_1(\overline{O(f)})+O(m^{d-1})\\
    &=m^d(d+1)\deg_L(X)\deg(\omega)\int_{X^{\mathrm{an}}_\omega} f\mu_{\overline L,\omega}+O(m^{d-1}).
\end{align*}

Thus we can fix an sufficiently large $m$ such that 
$$\widehat c_1(\overline L^{\ot m}\otimes \overline{O(f)})^{d+1}-C\cdot m^{d-1}>0$$

On the other hand, Proposition \ref{prop_compar_int} gives that 
\begin{align*}
    &\big|1/l_i^{d+1} c_1(\shfL_i^{\ot m}\ot \O(f))^{d+1}-\widehat c_1(\ovl L^{\ot m}\ot \ovl{O(f)})^{d+1})\big|\\
    &\kern 10em=\big| \widehat c_1(\ovl L_i^{\ot m}\ot \ovl{O(f)})^{d+1}-\widehat c_1(\ovl L^{\ot m}\ot \ovl{O(f)})^{d+1}\big|\\
    &\kern 10em\leqslant (d+1)\deg_{L^{\ot m}}(X)d(\ovl L_i^{\ot m},\ovl L^{\ot m})\\
    &\kern 10em = m^{d+1}(d+1) \deg_L(X)d(\ovl L_i,\ovl L).
\end{align*}

Hence for any $\epsilon>0$ small enough, there exists a normal model $(\pi:\scrX\rightarrow Y,\shfL)$ of $(X,L^{\ot l})$ for some $l\in \N_+$ such that 
\begin{itemize}
    \item[\textnormal(i)] $d\left(\overline L,(L,\{\lvert\cdot\rvert_{\shfL,\omega}^{1/l}\})\right)<\epsilon$,
    \item[\textnormal(ii)] $c_1(\shfL^{\ot m}\ot \O(f)^{\ot l})^{d+1}-l^{d+1}C\cdot m^{d-1}>0$,
    \item[\textnormal(iii)] $\shfL$ is nef.
\end{itemize}
We denote by $\mathscr N_{\epsilon}=\shfL^{\ot m}\ot \O(f)^{\ot l}$.
Due to Corollary \ref{coro_ineq_err}, \begin{align*}
    &\deg(\pi_*(\mathscr N_{\epsilon}^{\ot n}))\\
    &\kern 5em\geqslant \frac{c_1(\mathscr N_{\epsilon}^{\ot n})^{d+1}-l^{d+1}C\cdot m^{d-1}}{(d+1)!}n^{d+1}+o(n^{d+1}).
\end{align*}
Therefore for $n\gg0$, $$\deg(\pi_*(\mathscr N_\epsilon^{\ot n}))>\mathrm{rk}(\pi_*(\mathscr N_\epsilon^{\ot n}))(g-1),$$
where $g$ is the genus of $Y$.
The Riemann-Roch theorem shows that
$h^0(Y,\pi_*(\mathscr N_\epsilon^{\ot n}))>0.$
Hence there exists a non-zero section $s$ of $H^0(\scrX,\mathscr N_\epsilon^{\ot n})\simeq H^0(Y,\pi_*(\mathscr N_\epsilon^{\ot n}))$. Since $\{Z_\iota\}_{\iota\in I}$ is generic, there exists an $\iota_0\in I$ such that $Z_\iota\not\subset \mathrm{Supp}(s|_X)$ for every $\iota\geqslant \iota_0$, which yields that $[\overline{Z_\iota}]\cdot \mathrm{div}(s)$ is an effective cycle. The nefness of $\shfL$ gives that $$\begin{aligned}
&n c_1(\shfL)^{\dim Z_\iota}\cdot c_1((\shfL^{\ot m}\ot \O(f)^{\ot l}))\cdot [\ovl{Z_\iota}]\\
&\kern 1.5em
=c_1(\shfL)^{\dim Z_\iota}\cdot c_1((\shfL^m\ot \mathscr{O}(f)^{\ot l})^n)\cdot[\overline{Z_\iota}]=c_1(\shfL)^{\dim Z_\iota}\cdot \mathrm{div}(s)\cdot[\overline{Z_\iota}]\geqslant 0.
\end{aligned}$$
Let $\overline L'$ denote the adelic line bundle $\big(L,\{\lvert\cdot\rvert_{\shfL,\omega}^{1/l}\}\big)$. Then the inequality above can be interpreted as 
$$\widehat{c}_1(\overline L')^{\dim Z_\iota}\widehat c_1(\overline L'\ot\overline{O(f/m)})\cdot [Z_\iota]\geqslant0.$$
Consequently,
$$\begin{aligned}
&h_{\overline L}(Z_\iota)+\deg(\omega)\frac{\int_{X_\omega^\mathrm{an}}f/m \mu_{\overline L,Z_\iota,\omega}}{\dim Z_\iota+1}\\
&\kern 2em=\frac{\widehat{c}_1(\overline L)^{\dim Z_\iota}\widehat c_1(\overline L\ot\overline{O(f/m)})\cdot[Z_\iota]}{(\dim Z_\iota+1)\mathrm{deg}_L(Z_\iota)}\\
&\kern 2em \geqslant
\frac{\widehat{c}_1(\overline L')^{\dim Z_\iota}\widehat c_1(\overline L'\ot\overline{O(f/m)})\cdot [Z_\iota]}{(\dim Z_\iota+1)\deg_L(Z_\iota)}-d(\ovl L'|_{Z_\iota},\ovl L|_{Z_\iota})\\
&\kern 2em \geqslant -d(\ovl L',\ovl L)\geqslant -\epsilon.
\end{aligned}$$
Since $\epsilon$ is an arbitrary positive number and $h_{\overline L} (Z_\iota)\rightarrow 0$, we obtain that
$$\liminf\limits_{\iota\in I}\int_{X_\omega^{\mathrm{an}}}f\mu_{\overline L,Z_\iota,\omega}\geqslant 0$$
holds for any model function satisfying $\int_{X_\omega^{\mathrm{an}}}f \mu_{\overline L,\omega}>0$, which implies that our claim holds for continuous functions due to the uniform density of model functions.

Note that equation \eqref{eq_equidistribution} holds for any constant function. Now for any continuous function $f$ and $m>0$, let $$f_m=f-\int_{X^{\mathrm{an}}_\omega}f \mu_{\ovl L,\omega}-1/m.$$
Then $\displaystyle\int_{X_\omega^{\mathrm{an}}}f_m\mu_{\ovl L,\omega}>1/m>0$. Due to our claim, we have 
$\displaystyle\liminf\limits_{\iota\in I}\int_{X_\omega^{\mathrm{an}}}f_m\mu_{\overline L,Z_\iota,\omega}\geqslant 0,$ which implies that
$$\liminf\limits_{\iota\in I}\int_{X_\omega^{\mathrm{an}}}f\mu_{\overline L,Z_\iota,\omega}\geqslant \int_{X_\omega^{\mathrm{an}}}f \mu_{\overline L,\omega}+1/m.$$
As $m$ is arbitrary, we have 
$$\liminf\limits_{\iota\in I}\int_{X_\omega^{\mathrm{an}}}f\mu_{\overline L,Z_\iota,\omega}\geqslant \int_{X_\omega^{\mathrm{an}}}f \mu_{\overline L,\omega}.$$
Applying the same reasoning to $-f$, we get the reverse inequality, which concludes the proof.
\end{proof}

\subsection{Algebraic dynamical systems}
Let $X$ be a projective $K$-variety, $L$ an ample line bundle on $X$. Assume that there exists an endomorphism $\phi:X\rightarrow X$ such that $\phi^*L\simeq L^{\ot q}$ for some $q\in \N_{>1}$, which is called an \textit{algebraic dynamic system}. Let $\overline L_0=(L,\{\lvert\cdot\rvert_{0,\omega}\})$ be an adelic line bundle induced by some model $(\scrX,\shfL)$ of $(X,L)$. We define $\overline L_i=(L,\{\lvert\cdot\rvert_{i,\omega}\})$ inductively by
$$\lvert\cdot\rvert_{i+1,\omega}:=(\phi^*\lvert\cdot\rvert_{i,\omega})^{1/q}.$$
By \cite[Theorem 9.5.4]{bombieri2007heights}, $\{\overline L_i\}$ converges to an adelic line bundle $\overline L=(L,\{\lvert\cdot\rvert_{\mathrm{can},\omega}\})$, such that $\lvert\cdot\rvert_{\mathrm{can},\omega}:=(\phi^*\lvert\cdot\rvert_{\mathrm{can},\omega})^{1/q}$. Actually, it can be shown that $\overline L$ is a uniform limit of nef line bundles, hence semipositve. We define the \textit{canonical height} by $h_{\mathrm{can}}(\cdot):=h_{\overline L}(\cdot)$. Then $$h_{\mathrm{can}}(\phi(Z))=q h_{\mathrm{can}}(Z)$$ holds for any closed subvariety $Z$. We say $Z$ is \textit{preperiodic} if the set $\{Z,\phi(Z),\phi^2(Z),\cdots\}$ is finite, in which case we have $h_{\mathrm{can}}(Z)=0$. In particular, $h_{\mathrm{can}}(X)=0$. Therefore we can apply the preceeding equidistribution theorem to obtain the equidistribution theorem of generic preperiodic points with respect to an algebraic dynamical system. 

\bibliographystyle{plain}
\bibliography{mybibliography}
\end{document}